\documentclass[12pt,reqno]{amsart}
\usepackage[top=1in,bottom=1in,left=1.25in,right=1.25in]{geometry}
\usepackage{amssymb, amsmath,mathtools,mathabx}
\usepackage{mathrsfs}
\usepackage{amscd}
\usepackage{cjhebrew}
\usepackage[T1]{fontenc} 
\usepackage[utf8]{inputenc}
\usepackage{verbatim}
\usepackage{stmaryrd}
\usepackage{bbm}
\usepackage{todonotes}
\usepackage{thmtools}
\newcommand{\ConditionallyIndependent}[2]{#1 \perp\kern-5pt \perp #2}
\usepackage{enumitem}

\usepackage{enumerate}

\usepackage[colorlinks,linkcolor={blue},citecolor={blue},urlcolor={purple},]{hyperref}

\usepackage{tabularx}
\newcolumntype{L}{>{\arraybackslash}X}
\usepackage{multirow}

\theoremstyle{plain}
\newtheorem{theorem}{Theorem}[section]
\theoremstyle{remark}
\newtheorem{remark}[theorem]{Remark}

\theoremstyle{plain}
\newtheorem{corollary}[theorem]{Corollary}      
\newtheorem{lemma}[theorem]{Lemma}
\newtheorem{proposition}[theorem]{Proposition}

\theoremstyle{definition}
\newtheorem{definition}[theorem]{Definition}

\newtheorem{assumption}[theorem]{Assumption}

\numberwithin{equation}{section}

\numberwithin{equation}{section}

\renewcommand{\Re}{\operatorname{Re}}

\newcommand{\la}{\lambda}

\def\N{{\mathbb N}}

\def\R{{\mathbb R}}

\newcommand{\norm}[1]{\left\|#1\right\|}

\renewcommand{\S}{\mathbb{S}}

\newcommand{\HH}{\mathcal{H}} 

\newcommand{\eps}{\varepsilon}

\title[]{Self-similar blowup from arbitrary data for supercritical wave maps with additive noise
}

\author{Irfan Glogi\'c}
\address{Fakult\"at f\"ur Mathematik, Universit\"at Bielefeld, D-33501 Bielefeld, Germany}
\email{\href{mailto:irfan.glogic@uni-bielefeld.de}{irfan.glogic@uni-bielefeld.de}}

\author{Martina Hofmanov\'a}
\address{Fakult\"at f\"ur Mathematik, Universit\"at Bielefeld, D-33501 Bielefeld, Germany}
\email{\href{mailto:hofmanova@math.uni-bielefeld.de}{hofmanova@math.uni-bielefeld.de}}

\author{Eliseo Luongo}
\address{Fakult\"at f\"ur Mathematik, Universit\"at Bielefeld, D-33501 Bielefeld, Germany}
\email{\href{mailto:eluongo@math.uni-bielefeld.de}{eluongo@math.uni-bielefeld.de}}

\thanks{This research was funded in whole or in part by the Austrian Science Fund (FWF) 10.55776/PAT5825523.
	M.H. and E.L. are grateful for funding from the European Research Council (ERC) under the European Union’s Horizon 2020 research and innovation programme (grant agreement No. 949981) and for the financial support provided by the Deutsche Forschungsgemeinschaft (DFG, German Research Foundation) – Project-ID 317210226–SFB 1283.}

\usepackage{kantlipsum}

\usepackage{fullpage}
\allowdisplaybreaks
\begin{document}

	\begin{abstract}
		We consider stochastically perturbed wave maps from $\R^{1+d}$ into $\S^d$, in all energy-supercritical dimensions $d \geq 3$.
		We show that corotational non-degenerate Gaussian additive noise leads to self-similar blowup with positive probability for any corotational initial data. The same result without noise is conjectured, but unknown, for large data. 
	\end{abstract}
	
	\maketitle
	\maketitle
	\section{Introduction}
	\noindent The wave maps equation represents a generalization of the linear wave equation to maps that take values in Riemannian manifolds. More precisely, a map $\mathcal{U}$ from the flat Minkowski spacetime $(\R^{1+d},\eta)$ into a Riemannian manifold $( M,g)$ is called a \emph{wave map} if it is a critical point (under compactly supported perturbations) of the Lagrangian action functional
	\begin{equation*}
		\mathcal{S}[\mathcal{U}] = \frac{1}{2} \int_{\R^{1+d}} \big(-|\partial_t\mathcal{U}|^2_{g}+ \sum_{i=1}^{d}|\partial_i \mathcal{U}|_g^2 \big)d\mu_{\eta}.
	\end{equation*}
	By putting local coordinates on $M$, and using the Einstein summation convention, we arrive at the following, more compact, form\footnote{Here, Greek indices go from 0 to $d$ and Latin indices go from 1 to $n$, where $n$ is the dimension of the target manifold $M$. According to the Einstein summation convention, whenever an index variable appears twice in a single term, once as a subscript and once as a superscript, it implies a summation over the understood range of values for that index. We also note that the indices are raised and lowered according to the Minkowski metric $\eta = \text{diag}(-1,1,\dots,1)$; in particular $\partial^\alpha = \eta^{\alpha\beta}\partial_\beta $.} 
	\begin{equation*}
		\mathcal{S}[\mathcal{U}] = \frac{1}{2} \int_{\R^{1+d}} g_{ij}(\mathcal{U})\partial_\alpha \mathcal{U}^{i} \partial^\alpha \mathcal{U}^{j}d\mu_{\eta}.
	\end{equation*}
	The associated Euler-Lagrange equations consequently read as
	\begin{align}\label{intro_eq}
		\begin{cases}
			\,\partial^{\mu}\partial_{\mu} \mathcal{U}^i+\Gamma^i_{jk}(\mathcal{U}) \partial^{\mu}\mathcal{U}^j\partial_{\mu}\mathcal{U}^k=0, \quad i=1,\dots,n,\\
			\,\mathcal{U}(0,\cdot)=\mathcal{U}_0,\\
			\,\partial_t \mathcal{U}(0,\cdot)=\hat{\mathcal{U}}_{0},
		\end{cases}
	\end{align}
	where $n$ is the dimension of the manifold $M$, $\Gamma_{jk}^i$ are the Christoffel symbols associated with the metric $g$, and $\partial_t$, according to custom, stands for $\partial_0$.

	In contrast to the linear wave equation, the wave maps system \eqref{intro_eq} is inherently nonlinear due to the curvature of the target manifold $(M,g)$. From a dynamical perspective, a general goal is to understand the interplay between the geometry of the target  and the behavior of solutions to the associated Cauchy problem \eqref{intro_eq}. In particular, one of the central questions is that of finite-time blowup: 
	\begin{center}
		\emph{Can smooth and localized initial data $(\mathcal{U}_0,\hat{\mathcal{U}}_0)$ lead to singularity formation for \eqref{intro_eq} in finite time?}
	\end{center}
	Much research in the past fifty years has addressed this question, as we to some extent review in \autoref{Sec:History} below. In this section we mention only (some of) the works relevant for our main results below.
	We proceed by pointing out that, in the context of finite-time blowup, there is a useful guiding principle based on the \emph{energy functional}
	\begin{align}\label{Def:Energy}
		E[\mathcal{U}](t):=\frac{1}{2}\int_{\R^d} \delta^{\alpha\beta} g_{ij}(\mathcal{U}(t,X))\partial_{\alpha}\mathcal{U}^i(t,X)\partial_{\beta}\mathcal{U}^{j}(t,X) dX,
	\end{align}
	which is formally conserved by solutions of the system, and the scaling transform $\mathcal{U}(t,X) \mapsto \mathcal{U}^{\lambda}(t,X):=\mathcal{U}(t/\lambda,X/\lambda)$, which preserves the wave maps equation \eqref{intro_eq}. Namely, the functional $E$ transforms under the scaling $\mathcal{U} \mapsto \mathcal{U}^{\lambda}$ according to
	\begin{align*}
		E[\mathcal{U}^{\lambda}](t)=\lambda^{d-2} E[\mathcal{U}](t/\lambda),
	\end{align*}
	which shows that shrinking the solution locally (i.e., letting $\la \rightarrow 0^+$) leads to local energy dissipation precisely when $d \geq 3$. Singularity formation is therefore said to be energetically favourable in the so-called \emph{energy-supercritical} case, $d \geq 3$, and one therefore heuristically expects finite-time blowup, at least for some large initial data.
	We remark in passing that for the \emph{energy-subcritical} case, $d=1$, where local energy dissipation precludes concentration of solutions, compactness of the target ensures coercivity of the energy $E$, which in turn provides sufficient control over the wave maps flow to guarantee global regularity for any smooth and localized initial data.
	
	In this paper, we consider the energy-supercritical case of a canonical model of the wave maps equation, where the target manifold is the $d$-dimensional sphere $(\S^d, h)$ endowed with the standard round metric $h$ induced by the embedding into $\R^{d+1}$. For our analysis, it is convenient to work with the so-called \emph{normal coordinates} on $\S^d$. Namely, for a function $\mathcal{U}$ with values in $\S^d$, we consider the coordinate representation $\mathcal{U}=(\mathcal{U}^1,\dots,\mathcal{U}^d)$ where 
	\begin{equation*}
		\mathcal{U}^i := \theta \Omega^i, \quad \text{for} \quad i=1,\dots,d,
	\end{equation*}
	with $\theta$ being the polar angle, and $\Omega^i$ the coordinate functions of the embedding $\S^{d-1} \hookrightarrow \R^d$.
	Concerning existence of blowup for this model, it is known that the corresponding wave maps equation \eqref{intro_eq} admits, in all supercritical dimensions $d \geq 3$, an \emph{explicit} self-similar solution $\mathcal{U}_T$, which is given in the normal coordinates by
	\begin{equation}\label{Def:BizBie_sol}
		\mathcal{U}^i_T(t,X)= \Phi\left( \frac{|X|}{T-t} \right)X^i, \quad \text{with} \quad \Phi(\rho)=\frac{2}{\rho}\arctan\left(\frac{\rho}{\sqrt{d-2}}\right),
	\end{equation}
	and which forms singularity at the origin as $t \rightarrow T^-$.
	Originally identified by Turok and Spergel in \cite{turok1990global} for $d=3$, this solution was subsequently generalized by Bizo\'n and Biernat in \cite{bizon2015generic} to dimensions $d \geq 4$. We note that $\mathcal{U}_T$  belongs to the class of the so-called \emph{corotational} solutions
	\begin{align}\label{Def:Corrotationa_ansatz}
		\mathcal{U}^{i}(t,X)=u(t,\lvert X\rvert)X^i,\quad i=1,\dots, d.
	\end{align}
	In fact, for initial data of corotational form
	\begin{equation}\label{Def:Corot_data}
		\mathcal{U}_0(X)=u_0(|X|) X \quad \text{and} \quad  \hat{\mathcal{U}}_{0}(X)=\hat{u}_0(|X|)X,
	\end{equation}
	\eqref{intro_eq} reduces by means of \eqref{Def:Corrotationa_ansatz} to a Cauchy problem for a single $(d + 2)$-dimensional radial semilinear wave equation for the profile $u = u(t, r)$
	\begin{align}\label{intro_eq_2}
		\begin{cases}
			\,\partial_{t}^2 u-\partial_{r}^2 u-\dfrac{d+1}{r}\partial_{r} u=\dfrac{d-1}{2 r^3}\big(2r u-\sin (2r u)\big),\\
			\,u(0,\cdot)=u_0,\\
			\,\partial_t u(0,\cdot)=\hat{u}_0,
		\end{cases}    
	\end{align}
	with the corresponding self-similar blowup solution
	\begin{align}\label{blow_up_explicit}
		u_{T}(t,r)&:=\frac{1}{T-t}\Phi\left(\frac{r}{T-t}\right), \quad T>0,
	\end{align}
	where $\Phi$ is given in  \eqref{Def:BizBie_sol}.
	
	To understand the role that $u_T$ plays in the dynamics of \eqref{intro_eq_2}, numerical simulations have been performed, in \cite{BizChmTab00} for $d =3$, and in \cite{bizon2015generic} for $d \geq 4$. Surprisingly, these have revealed a striking universality in the blowup behavior: 
	\begin{center}
		\emph{For generic large corotational initial data, the solution blows up at the origin via the self-similar profile $u_T$.}
	\end{center}
	In other words, irrespective of the choice of initial data $(u_0,\hat{u}_0)$ leading to finite-time blowup, if one dynamically self-similarly rescales the solution as it approaches the blowup time, then the profile $\Phi$ emerges as an attractor of the rescaled evolution, i.e.,
	\begin{equation}\label{Eq:Rescaled_convergence}
		(T-t)u(t,(T-t)\cdot) \rightarrow \Phi
	\end{equation} 
	locally uniformly on $[0,+\infty)$ as $t \rightarrow T^-.$ This, in turn, translates for the corotational map $\mathcal{U}$ into
	\begin{equation}\label{Eq:Rescaled_convergence_corotational}
		\mathcal{U}(t,(T-t)\cdot) \rightarrow (\cdot)\Phi(|\cdot|),
	\end{equation}
	which takes place locally uniformly on $\R^d$ as $t \rightarrow T^-.$
	Following these observations, considerable effort has been devoted to their rigorous justification, starting with the pioneering works of Aichelburg, Donninger, and Sch\"orkhuber \cite{DonAic09,Don11,DonSchAic12}. This line of research, which relies on stability analysis localized to lightcones, eventually lead to a rigorous verification of, what could be called, a perturbative local version of the numerical observations above. Namely, for any fixed ball $B_R$ centered at the origin with radius $R>0$, all corotational data that are close enough to that of the blowup solution $u_T$, lead to finite-time blowup with \eqref{Eq:Rescaled_convergence_corotational} taking place uniformly on $B_R$; see \cite{CosDonXia16,CosDonGlo17,ChaDonGlo17,BieDonSch21,DonOst24}. This was then followed by the work of the first author \cite{glogic2025globally}, which takes the global-in-space stability point of view, and thereby removes the assumption of a pre-fixed ball $B_R$, hence establishing the convergence \eqref{Eq:Rescaled_convergence} globally on $[0,+\infty)$, and consequently obtains \eqref{Eq:Rescaled_convergence_corotational} uniformly on compact sets; see \cite[Theorem 1.1]{glogic2025globally} for a precise, quantitative, statement of this result. The non-perturbative counterpart of this result, i.e., the assertion that blowup is governed by $u_T$ for (generic) large data that are \emph{not necessarily close} to those of $u_T$, is, however, open. Motivated by this, in this paper we take a probabilistic point of view, that is, we consider the stochastic variant of \eqref{intro_eq} with additive corotational noise. We then show that for arbitrary corotational initial data, the corresponding solution does, in fact, blow up via \eqref{Def:BizBie_sol} with positive probability.
	
	\subsection{Wave maps with additive noise}
	We consider the following stochastic evolution equation on $\R^d$
	\begin{align}\label{eq_intro_noise}
		\begin{cases}
			\,\partial^{\mu}\partial_{\mu} \mathcal{U}^i+\Gamma^i_{j,k}(\mathcal{U})\partial^{\mu}\mathcal{U}^j\partial_{\mu}\mathcal{U}^k=X^i\dot{{\mathcal{W}}},\ i=1,\dots,d,\\
			\,\mathcal{U}(0,\cdot)=\mathcal{U}_0,\\
			\,\partial_t \mathcal{U}(0,\cdot)=\hat{\mathcal{U}}_{0},
		\end{cases}
	\end{align}
	for an infinite dimensional Brownian motion $\mathcal{W}$ on a filtered probability space  $(\Omega,\mathcal{F},(\mathcal{F}_t)_{t\geq 0},\mathbb{P})$ satisfying suitable assumptions described below; see \autoref{HP_noise_1} and \autoref{hp_noise_2}. To preserve the corotational structure of the solutions to \eqref{eq_intro_noise}, we assume that the noise $\mathcal{W}$  lies in the space of radially symmetric functions, and the initial data are corotational, i.e., \eqref{Def:Corot_data} holds for some radial $u_0, \hat{u}_0$. Then, similarly to above, by setting $ \mathcal{U}(t,X)=u(t,\lvert X\rvert)X$ and $n=d+2$, the system \eqref{eq_intro_noise} reduces to a single stochastic scalar wave equation\footnote{For $x=0$, the nonlinearity is interpreted in the limiting sense.}
	\begin{align}\label{wave_map_system}
		\begin{cases}
			\,\partial_{t}^2 u-\Delta u  =\dfrac{n-3}{2\lvert x\rvert^3}\big(2\lvert x\rvert u-\sin (2\lvert x\rvert u)\big)+\dot{\mathcal{W}},\\
			\,u(0,\cdot)= u_0,\\
			\,\partial_t u(0,\cdot)=\hat{u}_{0},
		\end{cases}
	\end{align}
	with radial data posed on $\R^n.$ Accordingly, we exclusively study the dynamics of solutions to \eqref{wave_map_system}, which we then ultimately translate back to the context of \eqref{eq_intro_noise}. 
	
	There are at least two additional motivations for studying the system \eqref{eq_intro_noise}. In realistic physical models, idealized equations such as the wave maps system inevitably neglect small-scale external perturbations, imperfections in the environment, and unresolved random fluctuations. A standard modeling approach to incorporate such effects is to introduce stochastic forcing terms. In the context of wave-type systems, additive noise has been employed to model, for instance, thermal fluctuations in plasma physics; see \cite{herr2023three}. Following this modeling philosophy, we consider perturbations by an additive noise term. The specific form of the noise, being proportional to the spatial variable $X$, is consistent with the imposed corotational symmetry and preserves the geometric structure of the problem. This choice ensures that the stochastic perturbation respects the corotational ansatz, and can be interpreted as modeling small, isotropic ambient forces. Moreover, additive noise introduces randomness independently of the solution itself, making it a clean model for external fluctuations that do not alter the underlying geometric structure. Consequently, additive noise acts as a ``non-resonant'' perturbation, allowing testing whether a mechanism of blowup is stable in a probabilistic sense. This is particularly important in supercritical problems, where the deterministic dynamics are often sensitive to small perturbations.

	The past forty years have seen a substantial amount of research on wave equations under stochastic forcing. Without aiming to be exhaustive, we refer to \cite{carmona1988random, carmona1988random2, dalang1998stochastic, peszat2000nonlinear, chow2002stochastic} and references therein for some of the earliest works on the subject. For wave equations taking values in Riemannian manifolds, we mention \cite{brzezniak2007strong, brzezniak2010stochastic}. In situations where the equation becomes singular and renormalization procedures are required, we refer to \cite{albeverio, gubinelli2018renormalization, gubinelli2022global, tolomeo}. Concerning local well-posedness of the stochastic equation \eqref{wave_map_system} (or, equivalently, \eqref{eq_intro_noise}) there are, to our knowledge, no currently available results. Accordingly, we first establish the existence and uniqueness of solutions in suitable function spaces, and derive the associated blowup alternative. We then proceed to analyze the stability of the blowup dynamics.
	
	\subsection{The main results}
	Due to the presence of noise, we have to work with a weak notion of corotational solutions to \eqref{eq_intro_noise}. To this end, we first establish a well-posedness theory for radial solutions to \eqref{wave_map_system}. For this, we adopt the functional framework developed by the first author in \cite{glogic2025globally} for the deterministic variant of \eqref{wave_map_system}, and consider the radial intersection homogeneous Sobolev spaces 
	\begin{equation*}
		\mathcal{H}^{s,k}_{rad}=\dot{H}^s_{rad}(\R^n)\cap \dot{H}^k_{rad}(\R^n)\times \dot{H}^{s-1}_{rad}(\R^n)\cap \dot{H}^{k-1}_{rad}(\R^n);
	\end{equation*} for a precise definition see \autoref{Notation} below. The basic solution concept we rely on is that of a \emph{mild solution}, as defined in, e.g., the classical monograph \cite{da2014stochastic}, to which we also refer for the basic notation in stochastic analysis used throughout the  paper. By denoting $\mathbf{u}(t):=(u(t,\cdot),\partial_t u(t,\cdot))$, we have the following well-posedness result, which we prove in \autoref{prel_physical}.
	\begin{theorem}\label{Thm_well_posed}
		Let $n\geq 5$ and $s,k>0$ such that
		\begin{align}\label{Cond_s_k}
			\frac{n}{2}-1<s<\frac{n}{2}-1+\frac{1}{2n-4},\ k>n,\ k\in \N.
		\end{align} 
		Under \autoref{HP_noise_1}, for each $\mathcal{F}_0$-measurable random variable $\mathbf{u}_0=(u_0,\hat{u}_{0})$ with values in $\HH^{s,k}_{rad}$, there exist a strictly positive stopping time $\tau^*(\mathbf{u}_0):\Omega\rightarrow (0,+\infty]$ and a progressively measurable process $\mathbf{u}:\Omega\times [0,+\infty)\rightarrow \HH_{rad}^{s,k}$ such that:
		\begin{itemize}[leftmargin=2em]
			\setlength{\itemsep}{1.5mm}
			\item[i)] For any positive stopping time $\tau < \tau^*(\mathbf{u}_0)$, the stopped process $\mathbf{u}(\cdot \wedge \tau)$ belongs to $C([0, +\infty); \HH_{rad}^{s,k})$ $\mathbb{P}$-a.s., and  is a mild solution of \eqref{wave_map_system} on $[0, \tau]$.
			\item[ii)] We have that 
			\begin{align*}
				\tau^*(\mathbf{u}_0,\omega)=+\infty \quad \text{or} \quad \limsup_{t\rightarrow \tau^*(\mathbf{u}_0,\omega)}\norm{\mathbf{u}(t)}_{\HH_{rad}^{s,k}}=+\infty\quad \mathbb{P}\text{-a.s.}
			\end{align*}    
		\end{itemize}
		Moreover, $\tau^*(\mathbf{u}_0)$ and $\mathbf u$ are unique in the following sense: if both $\left(\tau^{*}_1(\mathbf{u}_0), \mathbf{u}^1\right)$ and $\left(\tau^{*}_2(\mathbf{u}_0), \mathbf{u}^2\right)$ satisfy the conditions above, then $\tau^{*}_1(\mathbf{u}_0) = \tau^{*}_2(\mathbf{u}_0)$ $\mathbb{P}$-a.s., and for any stopping time $\tau < \tau^{*}_1(\mathbf{u}_0) \wedge \tau^{*}_2(\mathbf{u}_0)$, we have that
		\begin{align*}
			\mathbb{P}\left(\mathbf{u}^1(t \wedge \tau)=\mathbf{u}^2(t \wedge \tau)\ \forall t\geq 0\right)=1.
		\end{align*}
	\end{theorem}
	\begin{remark}\label{remark_numbers} 
		The condition \eqref{Cond_s_k} on the Sobolev exponents $s,k$ is imposed for several reasons. First, the assumption $s,k>\frac{n}{2}-1$ ensures that the associated heat semigroup in similarity variables exhibits exponential decay. Additionally, if $s$ is close enough to $\frac{n}{2}-1$, and $k$ is a large enough integer, the nonlinear operator in \eqref{wave_map_system} is  Lipschitz continuous in $\mathcal{H}_{rad}^{s,k}$. The two features above then allow for constructing local solutions by means of straightforward fixed point arguments, which, in particular, avoid having to use sophisticated dispersive equations tools like, e.g., Strichartz estimates.
	\end{remark}
	\begin{remark}
		To remove possible ambiguity, we provide below the exact definition of the mild solution referred to in $i)$; see \autoref{mild_solution}.
	\end{remark}
	\noindent 
	\noindent Now, by means of \autoref{Thm_well_posed} we can derive the analogous well-posedness result for the system \eqref{eq_intro_noise}. For this, we consider spaces
	\begin{align*}
	    \HH^{s,k}=\dot{H}^s(\R^d)\cap \dot{H}^k(\R^d)\times \dot{H}^{s-1}(\R^d)\cap \dot{H}^{k-1}(\R^d),
	\end{align*}
	and use the notion of corotational mild solutions to \eqref{eq_intro_noise}, for which we refer to \autoref{mild_solution} below.
	Now, in view of the equivalence between Sobolev norms of corotational maps and their radial profiles (see \cite[Proposition A.5, Remark A.6]{glogic2022stable}), we readily obtain the following local well-posedness result for \eqref{eq_intro_noise}. We use the notation $\boldsymbol{\mathcal{U}}(t)=(\mathcal{U}(t,\cdot),\partial_t \mathcal{U}(t,\cdot))$.
 	\begin{corollary}\label{corollary_well_posed}
		Let $d\geq 3$ and $s,k>0$ such that 
		\begin{align*}
			\frac{d}{2}<s<\frac{d}{2}+\frac{1}{2d},\ k>d+2,\ k\in \N.
		\end{align*}
		Under \autoref{HP_noise_1},
		for each $\mathcal{F}_0$-measurable corotational random variable $\boldsymbol{\mathcal{U}}_0=(\mathcal{U}_0,\mathcal{\hat{U}}_{0})$ with values in $\HH^{s,k}$, there exist a strictly positive stopping time $\tau^*(\boldsymbol{\mathcal{U}}_0):\Omega\rightarrow (0,+\infty]$ and a progressively measurable process $\boldsymbol{\mathcal{U}}:\Omega\times [0,+\infty)\rightarrow \HH^{s,k}$ such that:
		\begin{itemize}[leftmargin=2em]
			\setlength{\itemsep}{1mm}
			\item[i)]  For any positive stopping time $\tau < \tau^*(\boldsymbol{\mathcal{U}}_0)$, the stopped process $\boldsymbol{\mathcal{U}}(\cdot \wedge \tau)$ belongs to $C([0, +\infty); \HH^{s,k})$ $\mathbb{P}$-a.s., and is a corotational mild solution on $[0, \tau]$ of \eqref{eq_intro_noise}.
			\item[ii)] We have that
			\begin{align*}
				\tau^*(\boldsymbol{\mathcal{U}}_0,\omega)=+\infty \quad \text{or} \quad \limsup_{t\rightarrow \tau^*(\boldsymbol{\mathcal{U}}_0,\omega)}\norm{\boldsymbol{\mathcal{U}}(t)}_{\HH^{s,k}}=+\infty\quad \mathbb{P}\text{-a.s.}
			\end{align*}   
		\end{itemize}
		Moreover, uniqueness holds in the following sense: if both $\left(\tau^{*}_1(\boldsymbol{\mathcal{U}}_0), \boldsymbol{\mathcal{U}}^1\right)$ and $\left(\tau^{*}_2(\boldsymbol{\mathcal{U}}_0), \boldsymbol{\mathcal{U}}^2\right)$ satisfy the conditions above, then $\tau^{*}_1(\boldsymbol{\mathcal{U}}_0) = \tau^{*}_2(\boldsymbol{\mathcal{U}}_0)$ $\mathbb{P}$-a.s., and for any stopping time $\tau < \tau^{*}_1(\boldsymbol{\mathcal{U}}_0) \wedge \tau^{*}_2(\boldsymbol{\mathcal{U}}_0)$, it holds that
		\begin{align*}
			\mathbb{P}\left(\boldsymbol{\mathcal{U}}^1(t \wedge \tau)=\boldsymbol{\mathcal{U}}^2(t \wedge \tau)\ \forall t\geq 0\right)=1.
		\end{align*}
	\end{corollary}
	If the stopping time $\tau^*(\boldsymbol{\mathcal{U}}_0,\omega)$ above is finite, we say that the solution $\boldsymbol{\mathcal{U}}$ (or $\mathcal{U}$) \emph{blows up in finite time}, and we call $\tau^*(\boldsymbol{\mathcal{U}}_0,\omega)$ \emph{the blowup time}.
	Having a well-defined dynamics for \eqref{eq_intro_noise}, as well as the corresponding blowup alternative, we are ready to state the central result of the paper, which establishes the occurrence of self-similar blowup via $\mathcal{U}_T$ with positive probability for arbitrary data.
	\begin{theorem}\label{main_thm}
		Let $d\geq 3$ and $s,k>0$ such that 
		\begin{align*}
			\frac{d}{2}<s<\frac{d}{2}+\frac{1}{2d},\ k>d+2,\ k\in \N.
		\end{align*}
		Given \autoref{hp_noise_2}, for any corotational initial data
		\begin{equation}\label{Def:Init_data_theorem}
			\boldsymbol{\mathcal{U}}_0=(\mathcal{U}_0,\hat{\mathcal{U}}_0)\in \dot{H}^{s-1}(\mathbb{R}^d)\cap \dot{H}^{k+2}(\mathbb{R}^d)\times \dot{H}^{s-1}(\mathbb{R}^d)\cap \dot{H}^{k}(\mathbb{R}^d) 
		\end{equation}  
		and any positive number $\mathcal{T}$, the unique solution to \eqref{eq_intro_noise} in $\mathcal{H}^{s,k}$ given by \autoref{corollary_well_posed} blows up with positive probability via $\mathcal{U}_T$ for some $T<\mathcal{T}$. More precisely, for each choice of $\mathcal{U}_0$ given by \eqref{Def:Init_data_theorem} and $\mathcal{T}>0$, there exist a set $\mathcal{N} \subseteq \Omega$ of positive probability 
		such that for each $\omega\in{\mathcal{N}}$ the following hold:
		\begin{itemize}[leftmargin=2em]
			\setlength{\itemsep}{2mm}
			\item[i)] $\dfrac{\mathcal{T}}{2}<\tau^*(\boldsymbol{\mathcal{U}}_0,\omega) < \mathcal{T}$;
			\item[ii)] For all $t \in \big(\frac{\mathcal{T}}{2}, \tau^*(\boldsymbol{\mathcal{U}}_0,\omega)\big)$ we have that
			\begin{align*}
				\mathcal{U}(t,X)1_{\mathcal{N}}=\frac{X}{\tau^*(\boldsymbol{\mathcal{U}}_0,\omega)-t}&\Phi\left(\frac{\lvert X\rvert}{\tau^*(\boldsymbol{\mathcal{U}}_0,\omega)-t}\right)1_{\mathcal{N}}\\
				&+\Upsilon\left(t,\frac{\lvert X\rvert}{\tau^*(\boldsymbol{\mathcal{U}}_0,\omega)-t}\right)1_{\mathcal{N}}+ Xz_{\tfrac{\mathcal{T}}{2}}\big(t,\lvert X\rvert\big) 1_{\mathcal{N}},
			\end{align*}
			where $z_{\tfrac{\mathcal{T}}{2}}(t,|X|)$ is the solution of the linear wave equation with additive noise $\mathcal{W}$ (see \eqref{Eq:Stoch_lin_wave}) with zero initial data posed at $t=\tfrac{\mathcal{T}}{2}$, and on $\mathcal{N}$ the profile $\Upsilon$ satisfies
			\begin{align*}
				\norm{\Upsilon(t,\cdot)}_{\HH^{s,k}}\rightarrow 0\quad \text{as} \quad t \rightarrow \tau^*(\boldsymbol{\mathcal{U}}_0,\omega)^-. 
			\end{align*}
			In particular, on $\mathcal{N}$  we have that 
			\begin{align*}
				\mathcal{U}\big(t,(\tau^*(\boldsymbol{\mathcal{U}}_0,\omega)-t )\cdot\big)\rightarrow  (\cdot)\Phi(|\cdot|)  
			\end{align*}
			locally uniformly on $\mathbb{R}^d$ as $t\rightarrow \tau^*(\boldsymbol{\mathcal{U}}_0,\omega)^-.$
		\end{itemize}
	\end{theorem}
	\begin{remark}
		As will be evident from the proof, the value $\frac{\mathcal{T}}{2}$ in the statement above may be replaced with any  $\mathcal{T}_1 < \mathcal{T}$. We have opted to avoid this additional level of generality in order to simplify the notation.
	\end{remark}
	\subsection{On regularization by noise phenomena} Starting with the seminal works \cite{zvonkin1974transformation, veretennikov1981strong}, it has become well established that the addition of noise to (partial) differential equations can help prevent pathological behaviors such as non-uniqueness or blowup of solutions. We refer in particular to \cite{flandoli2010well,DP_regular_1, DP_regular_2, Priola, bertacco2023weak, beck2019stochastic, flandoli2021high, coghi2023existence, herr2023three,agresti2024global,crippa2025zero} for various instances of the so-called phenomenon of \emph{regularization by noise}, including examples arising in the context of fluid dynamics. The heuristic underlying these results is that pathological behaviors in (P)DEs typically occur only along exceptional trajectories and are not generic. Consequently, a small random perturbation can move the system away from them, thereby restoring uniqueness or preventing blowup. For a detailed exposition of this idea, we refer to the lecture notes \cite{flandoli2011random}. 
	
	In light of the results mentioned above and the preceding discussion, \autoref{main_thm} suggests that the blowup profile $\mathcal{U}_T$ (or $(\cdot)\Phi(|\cdot|)$) is not only in some sense generic but also that the corresponding blowup mechanism is so robust that the addition of noise, at least in the additive case, does not prevent the system's pathological behavior. On the contrary, it may even amplify it, as we are able to establish blowup for arbitrarily small initial data, for which the deterministic dynamics would otherwise be globally well-posed. From this perspective, our result is in line with those obtained for the Schrödinger equation in \cite{de2002effect, debussche2005blow}, where it was shown that the presence of smooth noise, whether additive or linear, provokes blowup with positive probability for \emph{generic} initial data. Unlike our work, however, \cite{de2002effect, debussche2005blow} do not provide asymptotic information on the behavior of solutions near the blowup time. We also point to works \cite{mueller1993blowup, mueller2000critical}, which investigate the effect of multiplicative noise on the heat equation, demonstrating that such noise can induce blowup of solutions with positive probability. Further examples of this phenomenon can be found in \cite{foondun2015non, foondun2019some, foondun2024instantaneous, Hocq19}.
	
	We conclude by noting that, as with all previously cited results on blowup induced by noise, we provide no information on the size of the probability of blowup. Furthermore, we cannot assert that every instance of blowup occurs in a self-similar fashion. Although we expect this to be the case, and moreover that $\mathbb{P}\left(\tau^*(\boldsymbol{\mathcal{U}}_0)<T\right)\rightarrow 1$ as $T\rightarrow +\infty$, we have no rigorous proofs.
	
	\subsection{Brief history of blowup for (deterministic) wave maps} \label{Sec:History}  
	Wave maps provide a geometric generalization of geodesics to higher dimensions. More precisely, the wave maps equation for maps from the flat Minkowski spacetime $(\R^{1+d},\eta)$ into a Riemannian manifold $(M,g)$ defines a motion of the Euclidean space $\R^d$ embedded into $M$, where in the extreme case, $d=0$, the system reduces to the geodesic equation. Unlike geodesics, however, wave maps in general exhibit rich dynamical behavior, including the formation of singularities in finite time. 
	
	For the energy-subcritical case, $d=1$, blowup is actually ruled out for compact target manifolds, regardless of their geometry. This simply follows  from the strong well-posedness of the one-dimensional wave maps at the energy regularity, together with the positive definiteness of the energy functional \eqref{Def:Energy}. For non-compact targets, however, there seems to be no known results on the existence of blowup so far.
	
	In the energy-critical case, $d=2$, the fine geometric properties of the target manifold become relevant, even in the compact case. For spherical targets, a remarkable series of works has addressed the existence and stability of blowup; see, e.g., \cite{KriSchTat08,RodSte10,RapRod12,KriMia20,KriMiaSch20}. In all these cases, the blowup mechanism is non-self-similar and corresponds to the so-called bubbling off of a harmonic map. In contrast, for targets of constant negative curvature, e.g., the hyperbolic plane, blowup is known to be impossible; see, e.g., \cite{Tao09,SteTat10,SteTat10a,KriSch12}.
	
	In the energy-supercritical case, $d \geq 3$, the construction of blowup is generally more accessible, as self-similar solutions often exist. The earliest results in this direction are works of Shatah and Tahvildar-Zadeh on rotationally symmetric targets with positive curvature \cite{Sha88,ShaTah94}. For targets of constant negative curvature, on the other hand, the existence of blowup remains a major open problem. However, when the assumption of constant curvature is dropped and the target is merely required to have negative curvature, self-similar blowup solutions have been constructed in sufficiently high dimensions; see, e.g., \cite{CazShaTah98,DonGlo19}.
	
	\subsection{Outline of the proof of \autoref{main_thm}}\label{sec_strategy}
	The proof is thematically divided between \autoref{sec_self_similar_analysis} and \autoref{Section_Irreducibility}. 
	The main result of \autoref{sec_self_similar_analysis} is \autoref{stability_lemma}, which establishes the nonlinear stability of the blowup mechanism of \eqref{Def:BizBie_sol} for the stochastic  system \eqref{eq_intro_noise}. Loosely speaking, we show that if the initial data $\boldsymbol{\mathcal{U}}_0$ are sufficiently close to $\boldsymbol{\mathcal{U}}_{T}(0)$ for some $T>0$, where $\boldsymbol{\mathcal{U}}_{T}=(\mathcal{U}_{T},\hat{\mathcal{U}}_T)$ denotes the explicit self-similar solution of \eqref{intro_eq}
	\begin{equation*}
		\mathcal{U}^i_T(t,X)= \Phi\left( \frac{|X|}{T-t} \right)\frac{X^i}{T-t}, \quad \hat{\mathcal{U}}^i_T(t,X):=\partial_t \mathcal{U}_T^i(t,X)= \hat{\Phi}\left( \frac{|X|}{T-t} \right)\frac{X^i}{(T-t)^2},
	\end{equation*}
	with profiles $\Phi$ and $\hat{\Phi}$ given by
	\begin{equation}\label{vector_self_similar_blow_up}
		\Phi(\rho)=\frac{2}{\rho}\arctan\left(\frac{\rho}{\sqrt{d-2}}\right)\quad  \text{and} \quad\hat{{\Phi}}(\rho)=\Phi(\rho)+\rho \Phi'(\rho)
	\end{equation}
	(equivalently, $\mathbf{{u}}_0:=(u_0,\hat{u}_0)$ is sufficiently close to $\mathbf{{u}}_{T}(0)$, where ${\mathbf{{u}}}_{T}:=(u_{T},\hat{u}_T)$ denotes the explicit self-similar solution of \eqref{intro_eq_2}
	\begin{align}\label{SS_sol}
    u_{T}(t,r)=\frac{1}{{T}-t}\Phi\left(\frac{r}{{T}-t}\right),\quad 
		\hat{u}_{T}(t,r)=\partial_t u_{T}(t,r)=\frac{1}{({T}-t)^2}\hat{\Phi}\left(\frac{r}{{T}-t}\right)\bigg),
	\end{align}
	then the corresponding solution $\mathcal{U}$ of \eqref{eq_intro_noise} given by \autoref{corollary_well_posed} blows up with a positive probability at time $\tilde{T}<2T$ by converging to $\mathcal{U}_{\tilde{T}}$. The proof of \autoref{stability_lemma} is based on the now celebrated Da Prato-Debussche trick \cite{da2002two}, combined with a deterministic stability analysis in the spirit of \cite{glogic2025globally}. 
	First, by the representation 
	\begin{equation}\label{Def:U_u}
		\mathcal{U}(t,X)=u(t,\lvert X\rvert)X
	\end{equation}
	and \cite[Lemma A.5]{glogic2022stable} we reduce the analysis to that of the dynamics of the radial profile $u$ evolving according to \eqref{wave_map_system}. Then, we define $z$ as the stochastic convolution starting from zero data for the linear wave equation \eqref{Eq:Stoch_lin_wave}. Consequently, as $z$ is a regular object, in order to study the blowup dynamics of \eqref{wave_map_system}, it suffices to analyze the behavior of
	\begin{equation}\label{Def:z}
		w := u - z
	\end{equation}
	 solving
	\begin{align}\label{eq_reminder_physical}
		\begin{cases}
			\,\partial_t^2 w-\Delta w=\dfrac{n-3}{2\lvert x\rvert^3}\big(2\lvert x\rvert (w+z)-\sin (2\lvert x\rvert (w+z))\big),\\
			\, w(0)= u_0,\\
			\, \partial_t w(0)=\hat{u}_{0},   
		\end{cases}
	\end{align}
	for  $x \in \R^n$. From this point onward, we adopt the standard approach of studying the flow near self-similar solutions by passing to \emph{similarity variables}. More precisely, for a parameter $\tilde{T}$, which is assumed to be close to $T$ and to be fixed later, we consider the change of independent variables
	\begin{equation*}
		\tau=\tau(t):=\log\bigg(\frac{\tilde{T}}{\tilde{T}-t}\bigg), \quad \xi=\xi(t,x):=\frac{x}{\tilde{T}-t}.
	\end{equation*}
	Furthermore, we define the new dependent variables (in their vector form)
	\begin{gather*}
		\mathbf{Z}(\tau,\xi):=(Z(\tau,\xi),\hat{Z}(\tau,\xi)), \;\; \text{for} \;\; Z(\tau,\xi)=(\tilde{T}-t)z(t,x) \;\, \text{and} \;\, \hat{Z}(\tau,\xi)=(\tilde{T}-t)^2\hat{z}(t,x),\\
		{\mathbf{W}}(\tau,\xi):=(W(\tau,\xi),\hat{W}(\tau,\xi)), \;\; \text{for} \;\; W(\tau,\xi)=(\tilde{T}-t)w(t,x) \;\, \text{and} \;\, \hat{W}(\tau,\xi)=(\tilde{T}-t)^2\hat{w}(t,x).
	\end{gather*}
	Throughout the paper, we use the `hat' notation to shortly denote the partial derivative with respect to the physical time $t$; in particular, above we have  $\hat{z} (t,x):=\partial_tz(t,x)$ and $\hat{w}(t,x):=\partial_tw(t,x)$.
	The advantage of the new coordinate frame is that  self-similar profiles now become \emph{static}, i.e., $\tau$-independent, and the problem of stability of self-similar blowup in finite time, becomes the one of \emph{asymptotic stability} of a steady-state solution (note that $\tau \rightarrow +\infty$ as $t \rightarrow \tilde{T}^-$).  In the new variables, the first-order vector formulation of \eqref{eq_reminder_physical} takes the form
	\begin{align}\label{DPD_self_similar}
		\begin{cases}
			\, \partial_\tau \mathbf{W}=\mathbf{L}_0 \mathbf{W}+\mathbf{n}(\mathbf{W}+\mathbf{Z}),\\
			\, \mathbf{W}(0,\cdot)=\mathbf{W}_0(\tilde{T}),
		\end{cases}    
	\end{align}
	where the linear operator $\mathbf{L}_0$ is given by
	\begin{align}\label{linear_operator_SS}
		\mathbf{L}_0&:=\begin{bmatrix}
			-1-\xi \cdot\nabla ,& 1\\ \Delta& -2-\xi\cdot\nabla
		\end{bmatrix},
	\end{align}
	the nonlinearity $\mathbf{n}$ is defined by        
	\begin{align} 
		\mathbf{n}(\mathbf{K})&:=(0,n_0(\mathbf{K})),\;\, \text{where} \;\,
		n_0(\mathbf{K})=\frac{n-3}{2 | \cdot |^3}\left(2| \cdot | K-\sin(2| \cdot | K)\right) \;\, \text{for} \;\, \mathbf{K}=(K,\hat{K}) \label{nonlinear_operator_physical},
	\end{align}
	and the trace of $\tilde{T}$ shows up only in the initial data in the following manner
	\begin{equation*}
		\mathbf{W}_0(\tilde{T}):=\big(  {\tilde{T}} w(0,{\tilde{T}}\cdot),{\tilde{T}^2}\hat{w}(0,{\tilde{T}}\cdot)  \big)=\big(  {\tilde{T}} u_0({\tilde{T}}\cdot),{\tilde{T}^2}\hat{u}_0({\tilde{T}}\cdot)  \big).
	\end{equation*}
	We then proceed by analyzing stability of the static profile
	\begin{equation*}
		\mathbf{\Phi}:=(\Phi,\hat{\Phi})
	\end{equation*}
	under small perturbations governed by \eqref{DPD_self_similar}.  Since the system is semilinear, we adopt a perturbative approach based on spectral stability. In short, what we do is the following. We consider the ansatz
	\begin{equation*}
	    \mathbf{W}(\tau,\cdot)=\mathbf{\Phi}+\mathbf{\Psi}(\tau,\cdot),
	\end{equation*}
	linearize around $\mathbf{\Phi}$, and then use dynamical systems tools such as fixed point arguments and Lyapunov-Perron methods, along with spectral stability results from \cite{glogic2025globally}, to prove that for any $(u_0,\hat{u}_0)$ close enough to $\mathbf{{u}}_{T}(0)$ and $\tilde{T}$ close enough to $T$, the initial data $\mathbf{\Psi}(0,\cdot)$ is small, and there is, furthermore, a specific choice of $\tilde{T}$, possibly also depending on the realization of the noise, such that the corresponding profile $\mathbf{\Psi}$ is global and exponentially decaying. Then, by undoing the similarity variables, we get that  the solution $w$ to \eqref{eq_reminder_physical} can be written for $t \in [0,\tilde{T})$ in the following form
	\begin{align}\label{eq:ansatz_sol}
		w(t,x)=\frac{1}{\tilde{T}-t}\Phi\left(\frac{\lvert x\rvert}{\tilde{T}-t}\right)+\frac{1}{\tilde{T}-t}{\Psi}\bigg(\log\bigg(\frac{\tilde{T}}{\tilde{T}-t}\bigg),\frac{\lvert x\rvert}{\tilde{T}-t}\bigg),
	\end{align}
	where ${\Psi}(\tau,\cdot)$ goes to $0$ in $\HH_{rad}^{s,k}$ as $\tau\rightarrow +\infty$, for such realization of $\omega$. Consequently, we conclude from the embedding $\mathcal{H}^{s,k}_{rad} \hookrightarrow L^\infty$ that
	\begin{equation*}
		(\tilde{T}-t)w(t,(\tilde{T}-t)\cdot) \rightarrow \Phi
	\end{equation*}
	globally uniformly on $\R^n$ as $t \rightarrow \tilde{T}^-$; cf.~numerical findings \eqref{Eq:Rescaled_convergence}.
	Consequently, by \eqref{Def:z} the same applies to $u$, and by \eqref{Def:U_u} to $\mathcal{U}$ as well, though only locally uniformly.
		Now we provide a more detailed overview of the above summarized argument, and concentrate the discussion around the novelties brought about by the presence of noise.

Linearization around $\mathbf{\Phi}$ leads to the following evolution equation for $\mathbf{\Psi}$
	\begin{align}\label{stability_system_ss}
		\begin{cases}
			\partial_\tau \mathbf{\Psi}=\mathbf{L} \mathbf{\Psi}+\mathbf{N}(\mathbf{\Psi}+\mathbf{Z})+\mathbf{V}\mathbf{Z},\\
			\mathbf{\Psi}(0)=\mathbf{\Psi}_{0,\tilde{T}},
		\end{cases}
	\end{align}
	where 
	\begin{align}\label{initial_conditions_self_sim_perturb}
		\mathbf{\Psi}_{0,\tilde{T}}=({\Psi}_{0,\tilde{T}}, \hat{\Psi}_{0,\tilde{T}})=\left(\tilde{T}u_0(\tilde{T}\cdot)-\Phi,\tilde{T}^2\hat{u}_0(\tilde{T}\cdot)-\hat{\Phi}\right)
	\end{align}
	and
	\begin{gather}
		\mathbf{L}:=\mathbf{L}_0+\mathbf{V},\quad \text{with} \quad \mathbf{V}\mathbf{K}:=
		\begin{bmatrix}\label{linear_perturbed_operators_SS} 0 \\
			\frac{8(n-4)(n-3)K}{(\lvert\xi\rvert^2+n-4)^2}\end{bmatrix}
		 \quad \text{and} \\
		\mathbf{N}(\mathbf{K}):=
		\begin{bmatrix}
			0\\
			\frac{n-3}{2\lvert \xi\rvert^3}\left(\gamma(\lvert \xi\rvert(\Phi+K))-\gamma(\lvert \xi\rvert \Phi)-\gamma'(\lvert \xi\rvert \Phi)(K)\right)
		\end{bmatrix},\quad \gamma(y):=2y-\sin(2y)\label{nonlinear_perturbed_operators_SS}.
	\end{gather}
	Note that at least two differences arise between \eqref{stability_system_ss} and the system treated in \cite{glogic2025globally}. One, which is quite minor, is that here we do not restrict ourselves to the choice of $T=1$ and treat the general case of $T>0$. The second, and much more fundamental, difference is the appearance of $\mathbf{Z}$-terms in \eqref{stability_system_ss} which are due to the presence of noise in \eqref{wave_map_system}. We note that, in general, even when the noise $\mathcal{W}$ is smooth, it acts as an external forcing that may break the stability of the self-similar profile $\mathbf{\Phi}$. Nevertheless, under suitable non-degeneracy assumptions on $\mathcal{W}$, we show that this scenario does not occur with positive probability. In particular, there exists a $\tilde{T}\in \bigl(\tfrac{T}{2},\tfrac{3T}{2}\bigr)$, depending on the realization of the noise $\mathcal{W}$ and on the initial condition $\mathbf{u}_0$, such that one can construct a global solution of \eqref{stability_system_ss} which decays sufficiently fast to $0$ in $\HH_{rad}^{s,k}$ as $\tau \to +\infty$.
	
	The guiding principle of the argument sketched above, in analogy with \cite{glogic2025globally}, is that small perturbations in $\HH_{rad}^{s,k}$ of the initial profile $\mathbf{u}_{T}(0)$ still lead to blowup with positive probability, namely on those realizations where the noise does not dominate the dynamics. We also point out that, a priori, the solutions constructed via the ansatz \eqref{eq:ansatz_sol} may exhibit poor measurability properties, since for each realization of the noise the parameter $\tilde{T}$ depends on the entire realization of $\mathcal{W}$, at least up to time $\tfrac{3}{2}T$. However, due to the uniqueness of $\mathbf{w}$ solving \eqref{fixed_point_problem_v}, it suffices to show that objects of the form \eqref{eq:ansatz_sol} also solve \eqref{fixed_point_problem_v}, thereby recovering the corresponding measurability properties. Moreover, by the definition of the stopping time $\tau^*(\mathbf{u}_0,\omega)$, it coincides with $\tilde{T}$ on the subset of $\Omega$ identified by this construction.
	
	The main result of \autoref{Section_Irreducibility}, stated in \autoref{Thm_irreducibility}, informally asserts that if $\mathbf{u}_0$ is sufficiently regular, then for any $\mathcal{T}, \varepsilon > 0$ and any $\mathbf{u}_1 \in \HH_{rad}^{s,k}$, the solution provided by \autoref{Thm_well_posed} has, with positive probability, no blowup before time $\mathcal{T}$ and, at time $\mathcal{T}$, lies in the open ball of radius $\varepsilon$ in $\HH_{rad}^{s,k}$ centered at $\mathbf{u}_1$. This is an irreducibility result for \eqref{wave_map_system} reflecting some ideas from \cite{de2002effect, flandoli1997irreducibility}. To implement our strategy, we first show that if $\mathbf{u}_1$ is sufficiently smooth, then there exists a smooth function $f$ such that the solution of
	\begin{align*}
		\begin{cases}
			\partial_t^2 u-\Delta u=\dfrac{n-3}{2\lvert x\rvert^3}(2\lvert x\rvert u-\sin (2\lvert x\rvert u))+\partial_t f,\\
			u(0)= u_0,\\
			\partial_t u(0)=\hat{u}_0  
		\end{cases}
	\end{align*}
	coincides with $\mathbf{u}_1$ at time $\mathcal{T}$. Secondly, we establish a form of continuity of the solution map of the above system with respect to the controllers $f$. This is achieved by once again exploiting the Da Prato–Debussche trick: we introduce $\mathbf{z}$ as the solution of the linear wave equation \eqref{linear_wave_evolution} with forcing term $\partial_t f$, and $\mathbf{w}$ as the solution of \eqref{eq_reminder_physical} corresponding to this choice of $\mathbf{z}$. Then we show that if there already exists a solution of the deterministic system \eqref{eq_reminder_physical} up to a deterministic time $\mathcal{T}$, then small perturbations of either the initial condition $\mathbf{u}_0$ or the controller $\mathbf{z}$ still permit the construction of solutions to \eqref{eq_reminder_physical} up to time $\mathcal{T}$. Moreover, these solutions depend continuously on the data in $C([0,\mathcal{T}]; \HH_{rad}^{s,k})$. With these two results in hand, the irreducibility of \eqref{wave_map_system} reduces to a support theorem for $\mathbf{z}$, where we once again make essential use of the non-degeneracy assumptions on the noise $\mathcal{W}$.
	By combining \autoref{stability_lemma} with \autoref{Thm_irreducibility}, we conclude the argument leading to \autoref{main_thm}.
	
	Let us end the section by commenting on the restriction on the Sobolev exponents $s,k$ 
	\begin{align*}
		\frac{d}{2}<s<\frac{d}{2}+\frac{1}{2d},\ k>d+2,\ k\in \N,
	\end{align*}
	appearing in \autoref{Thm_well_posed}, \autoref{corollary_well_posed} and \autoref{main_thm}. As already mentioned in \autoref{remark_numbers}, it guarantees that the nonlinearity in \eqref{wave_map_system} is locally Lipschitz in $\HH_{rad}^{s,k}$ both in physical and self-similar variables (see \autoref{lemma_regularity_nonlinear} below and \cite[Lemma 7.1]{glogic2025globally}), and that the semigroup generated by the wave operator $\mathbf{L}_0$ exhibits exponential decay in $\HH_{rad}^{s,k}$; see  \cite[Proposition 4.1, Proposition 6.2]{glogic2025globally}. Additionally, if $\mathbf{u}=(u,\hat{u}) \in \HH_{rad}^{s,k}$, then $u$ (resp. $\hat{u}$) and its derivatives up to order $\lceil \tfrac{n}{2} \rceil - 1$ (resp. $\lceil \tfrac{n}{2} \rceil - 2$) are uniformly continuous functions. Therefore, $\boldsymbol{\mathcal{U}}(X) = \mathbf{u}(|X|)X$ is a well-defined object in $\HH^{s,k}$, \emph{cf.} \cite[Proposition A.5, Remark A.6]{glogic2022stable}, and it is possible to assign a precise meaning to \eqref{eq_intro_noise}.
	\subsection*{Frequently used conventions}The reader might feel disoriented by the simultaneous consideration of \eqref{eq_intro_noise} and \eqref{wave_map_system}, as well as the fact we work with the former in both physical and self-similar variables. In addition, the combination of the Da Prato–Debussche trick and the self-similar ansatz, as described in \autoref{sec_strategy}, requires the introduction of several additional objects for our analysis. To assist the reader, we summarize our conventions here, providing a more compact overview of the notation used throughout the paper.
	\begin{center}
		\begin{tabular}{||c | c | c ||} 
			\hline
			Object & Physical variables  & Self-Similar variables \\ [0.5ex] 
			\hline\hline Solution of wave map & $\boldsymbol{\mathcal{U}}=(\mathcal{U},\hat{\mathcal{U}})$ & N/A
			\\ \hline 
			\shortstack{Solution of wave map with\\ corotational symmetry}
			& $\mathbf{u}=(u,\hat{u})$ & $\mathbf{U}=(U,\hat{U})$
			\\  \hline Stochastic convolution & $\mathbf{z}=(z,\hat{z})$  & $\mathbf{Z}=(Z,\hat{Z})$ \\ 
			\hline
			Self-Similar Profile & $\mathbf{u}_{T}=(u_{T},\hat{u}_{T})$  & $\mathbf{\Phi}=(\Phi,\hat{\Phi})$ \\
			\hline
			Da Prato-Debussche ansatz & $\mathbf{u}=\mathbf{w}+\mathbf{z},\  \mathbf{w}=(w,\hat{w}) $  & $\mathbf{U}=\mathbf{W}+\mathbf{Z},\ \mathbf{W}=(W,\hat{W}) $ \\
			\hline
			Self-Similar ansatz & N/A  & $\mathbf{W}=\mathbf{\Phi}+\mathbf{\Psi},\ \mathbf{\Psi}=(\Psi,\hat{\Psi}) $\\
			\hline
			Wave operator & $\mathbf{A}$  & $\mathbf{L}_0$ \\
			\hline
			Wave Semigroup & $\mathbf{T}(t)$  & $\mathbf{S}_0(\tau)$ \\
			\hline
			Nonlinearity & $\mathbf{n}=(0,n_0)$, see \eqref{nonlinear_operator_physical}  & $\mathbf{N},$ see \eqref{nonlinear_perturbed_operators_SS} \\ 
			\hline
			\shortstack{Linearization of the nonlinearity \\ around the self-similar profile} & N/A & $\mathbf{V}\cdot=\left(0,
			\frac{8(n-4)(n-3)}{(\lvert\xi\rvert^2+n-4)^2}\cdot\right)$ \\
			\hline \shortstack{Perturbed wave operator} & N/A & $\mathbf{L}$ \\
			\hline \shortstack{Perturbed wave semigroup} & N/A & $\mathbf{S}(\tau)$ \\ \hline
		\end{tabular}
	\end{center}
	Due to the vectorial nature of the wave equation when viewed as an evolution equation, we frequently deal with vector-valued quantities, where the second component is closely related to the time derivative of the first component.\footnote{It corresponds exactly to the time derivative in physical variables, but this is not the case in self-similar variables.} As indicated in the table, we adopt the convention that bold letters denote vector-valued quantities, while unbold letters represent scalar quantities. When it is necessary to refer to specific components, we use the corresponding unbold letter for the first component and the same letter with a hat for the second component.
	
	\subsection{Notation}\label{Notation}
	Let $m\geq 3$. We denote by $C_c^{\infty}(\R^m)$ the standard space of smooth and compactly supported test functions on $\R^m$, $\mathscr{S}(\R^m)$ the space of Schwartz functions on $\R^m$ and $L^2(\R^m)$ the space of square integrable functions. We denote by $\langle\cdot,\cdot\rangle$ (resp. $\norm{\cdot}$) the inner product (resp.~norm) on $L^2(\R^m)$. We also need to introduce the space 
	\begin{align*}
		C_{c,rad}^\infty(\R^m) &:= \{ f : [0,\infty) \rightarrow \mathbb{R} ~ \vert ~ f(\lvert\cdot \rvert) \in C_c^{\infty}(\R^m)  \},\\
		\mathscr{S}_{rad}(\R^m) &:= \{ f : [0,\infty) \rightarrow \mathbb{R} ~ \vert ~ f(\left|\cdot\right|) \in \mathscr{S}(\R^m)  \}.
	\end{align*}
	We will work with Sobolev spaces, both homogeneous and non-homogeneous variants. We denote by $H^r_{rad}(\R^m)$ the closure of $C_{c,rad}^\infty(\R^m)$ with respect to the standard $H^r(\R^m)$ topology. Consequently, we define the product radial Sobolev spaces
	\begin{align*}
	    \mathbf{H}^{r}(\R^m):={H}_{rad}^{r}(\R^m)\times {H}_{rad}^{r-1}(\R^m)
	\end{align*}
	with the corresponding inner product
	\begin{align*}
	    \langle \mathbf{u},\mathbf{v}\rangle_{\mathbf{H}^{r_1}(\R^m)}:=\langle u_1,v_1\rangle_{{H}^{r_1}_{rad}(\R^m)}+\langle u_2,v_2\rangle_{{H}^{r_1-1}_{rad}(\R^m)},
	\end{align*}
	for $\mathbf{u}=(u_1,u_2)$ and $\mathbf{v}=(v_1,v_2).$
	
	 Concerning homogeneous Sobolev spaces, for $r > -\frac{m}{2}$, we define the inner product for test functions $u,v\in C^{\infty}_c(\R^m)$
	\begin{align*}
		\langle u,v\rangle_{\dot{H}^r(\R^m)}:=\langle  \lvert \cdot\rvert^r \mathscr{F}(u), \lvert \cdot\rvert^r \mathscr{F}(v)\rangle,
	\end{align*}
	$\mathscr{F}$ being the Fourier transform. Consequently, we have the norm $\norm{u}_{\dot{H}^r(\R^m)}^2=\langle u,u\rangle_{\dot{H}^r(\R^m)}.$ It is well known that if $r\in \N$ then
	\begin{align*}
		\norm{u}_{\dot{H}^r(\R^m)}^2\simeq \sum_{\lvert \alpha\rvert=r}\norm{ \partial_{\alpha}u}^2.
	\end{align*}
	We furthermore denote by $\dot{H}^r(\R^m)$ (resp.~$\dot{H}^r_{rad}(\R^m)$) the closure of $C^{\infty}_c(\R^m)$ (resp.~$C^{\infty}_{c,rad}(\R^m)$) with respect to $\norm{\cdot}_{\dot{H}^r(\R^m)}$.
	 For $r_2\geq r_1>-\frac{m}{2}$, we define the intersection spaces 
	\begin{equation*}
	    \dot{H}_{rad}^{r_1,r_2}(\R^m):=\dot{H}^{r_1}_{rad}(\R^m)\cap \dot{H}_{rad}^{r_2}(\R^m),
	\end{equation*}
	and, in case of $r_2\geq r_1>-\frac{m}{2}+1$, the following Cartesian product thereof,
	\begin{equation*}
	    \HH_{rad}^{r_1,r_2}(\R^m):=\dot{H}_{rad}^{r_1}(\R^m)\cap \dot{H}_{rad}^{r_2}(\R^m)\times \dot{H}_{rad}^{r_1-1}(\R^m)\cap \dot{H}_{rad}^{r_2-1}(\R^m)
	\end{equation*}
	with respective inner products
	\begin{gather*}
		\langle u,v\rangle_{\dot{H}^{r_1}_{rad}(\R^m)\cap \dot{H}_{rad}^{r_2}(\R^m)}:=\langle u,v\rangle_{\dot{H}^{r_1}_{rad}(\R^m)}+\langle u,v\rangle_{\dot{H}^{r_2}_{rad}(\R^m)},\\
        \langle \mathbf{u},\mathbf{v}\rangle_{\HH_{rad}^{r_1,r_2}(\R^m)}:=\langle u_1,v_1\rangle_{\dot{H}^{r_1}_{rad}(\R^m)\cap \dot{H}_{rad}^{r_2}(\R^m)}+\langle u_2,v_2\rangle_{\dot{H}^{r_1-1}_{rad}(\R^m)\cap \dot{H}_{rad}^{r_2-1}(\R^m)},
    \end{gather*}
    for $u,v\in \dot{H}^{r_1}_{rad}(\R^m)\cap \dot{H}_{rad}^{r_2}(\R^m)$ and $\mathbf{u}=(u_1,u_2),\mathbf{v}=(v_1,v_2)\in \HH_{rad}^{r_1,r_2}(\R^m)$.
	Moreover, for each $T_1>0,$ $\alpha\in (0,\frac{1}{2}),\ p\in (2,+\infty)$ with $\alpha p>1$, we define 
    \begin{gather*}
		\mathcal{E}^r_{0,T_1}:=\left\{\mathbf{z}=(z,\hat z)\in C([0,T_1];\mathbf{H}^{r+\frac{n+1}{2}}(\R^m)): z(0)=\hat{z}(0)=0,\ z(t)=\int_0^t \hat{z}(s)ds \right\},\\
		\mathcal{W}^{\alpha,p,r}_{0,T_1}:=\left\{f\in W^{\alpha,p}(0,T_1;H^{r+\frac{n+1}{2}}_{rad}(\R^m)): f(0)=0\right \},
	\end{gather*} endowed with the strong topology of $C([0,T_1];\mathbf{H}^{r+\frac{n+1}{2}}(\R^m))$ and $W^{\alpha,p}(0,T_1;H^{r+\frac{n+1}{2}}_{rad}(\R^m))$ respectively. The equality $z(t)=\int_0^t \hat{z}(s)ds$ has to be understood in the sense of tempered distributions $\mathscr{S}'$. Obviously, $\mathcal{E}_{0,T_1}^r, \mathcal{W}^{\alpha,p,r}_{0,T_1} $ are separable Banach spaces. Finally, we adopt the usual asymptotic notation $a \lesssim b$ to denote $a \leq Cb$ for some constant $C>0$. Furthermore, we write $a \simeq b$ if both $a \lesssim b$ and $b \lesssim a$ hold.
	For simplicity, in the following we do not explicitly indicate the dependence of our function spaces on the underlying spatial dimension, unless it differs from $n=d+2$.
    
\section{Preliminaries}	\label{sec:Stoch_well-posednes}
\noindent This section collects several auxiliary results needed for the proof of \autoref{main_thm}. In \autoref{prel_physical} we introduce the basic notation for the linear wave equation and provide proofs of \autoref{Thm_well_posed} and a support theorem for the stochastic convolution originating from zero initial data. \autoref{prel_self_similar} then turns to the analysis in self-similar variables, and recalls the main elements of the linearized and spectral framework introduced in \cite{glogic2025globally}.
	
\subsection{Analysis in physical variables:~well-posedness and stochastic convolution}\label{prel_physical}
Denote by $\mathbf{T}(t)$  the semigroup associated with the linear wave equation and by $\mathbf{A}$ its infinitesimal generator. According to \cite[Chapter 7.4]{pazy2012semigroups}, these operators are well defined both in homogeneous and non-homogeneous Sobolev spaces and preserve radial symmetry. Hence, for each $r_2\geq r_1> -\frac{n}{2}+1$, $\mathbf{T}(t)$ is a bounded linear operator on $\HH_{rad}^{r_1,r_2}$ (resp. $\mathbf{H}^{r_1}$) and its infinitesimal generator 
    \begin{align*}
    \mathbf{A}:D(\mathbf{A})\subseteq  \HH_{rad}^{r_1,r_2}\rightarrow \HH_{rad}^{r_1,r_2} \quad \text{(resp. } \mathbf{A}:D(\mathbf{A})\subseteq  \mathbf{H}^{r_1}\rightarrow \mathbf{H}^{r_1})
    \end{align*}
satisfies $D(\mathbf{A})=\HH_{rad}^{r_1,r_2+1}$ (resp. $D(\mathbf{A})=\mathbf{H}^{r_1+1})$. For smooth initial data $\mathbf{u}_0=(u_0,\hat{u}_0),$  the function $ \mathbf{u}(t)=(u(t),\hat{u}(t))=\mathbf{T}(t)\mathbf{u}_0$ is the unique solution of the wave equation
	\begin{align}\label{linear_wave_1}
		\begin{cases}
			\partial_t^2u=\Delta u,\\
			u(0)=u_0,\\
			\partial_tu(0)=\hat{u}_0,  
		\end{cases} 
	\end{align}
	that is, $\mathbf{u}(t)$ satisfies the first-order evolution system
\begin{align}\label{linear_wave_evolution}
		\begin{cases}
			\partial_t u=\hat{u},\\
			\partial_t \hat{u}=\Delta u,\\
			u(0)=u_{0},\\
			\hat{u}(0)=\hat{u}_0.
		\end{cases}
	\end{align}
	Consider a complete filtered probability space $(\Omega,\mathcal{F},(\mathcal{F}_t)_{t\geq 0},\mathbb{P})$ with right-continuous filtration, and fix $s, k$ as in \autoref{Thm_well_posed} and \autoref{main_thm}. The Brownian motion $\mathcal{W}$ appearing in \eqref{wave_map_system} is assumed to satisfy different, increasingly restrictive conditions depending on whether the focus is on local well-posedness, i.e., \autoref{Thm_well_posed} and \autoref{corollary_well_posed}, or also on the blowup properties, i.e., \autoref{main_thm}. In the former case, the following assumption is imposed.
	\begin{assumption}\label{HP_noise_1}
		Let $\mathcal{W}$ be an infinite-dimensional Brownian motion adapted to $(\mathcal{F}_t)_{t \ge 0}$ with values in $\dot{H}^{s-1,k-1}_{\mathrm{rad}}$.  
	\end{assumption}
	More explicitly, $\mathcal{W}$ is $\mathbb{P}$-a.s.~continuous process taking values in $\dot{H}^{s-1,k-1}_{\mathrm{rad}}$: if $U$ is a separable Hilbert space and $B_t$ a cylindrical Brownian motion on $U$, there exists $J \in L(U;\dot{H}^{s-1,k-1}_{\mathrm{rad}})$ such that $JJ^*$ has finite trace and 
\begin{align}\label{formulation_2_noise}
  \mathcal{W}_t = J B_t.
\end{align}
Under these assumptions on $\mathcal{W}$, it follows from \cite{da2014stochastic} that the stochastic wave equation
\begin{align}\label{Eq:Stoch_lin_wave}
\begin{cases}
\partial_t^2 z = \Delta z + \dot{\mathcal{W}},\\
z(0) = 0,\\
\partial_t z(0) = 0,
\end{cases}
\end{align}
admits a unique mild (equivalently, weak) solution $\mathbf{z} \in C([0,+\infty); \HH_{\mathrm{rad}}^{s,k})$, which satisfies
\begin{align}\label{stochastic_wave}
\begin{cases}
d\mathbf{z} = \mathbf{A}\mathbf{z}\,dt + d\mathbb{W}_t,\\
\mathbf{z}(0) = 0,
\end{cases}
\qquad \mathbb{W}_t = (0, \mathcal{W}_t),
\end{align}
and is given by the mild formulation
	\begin{align*}
	    \mathbf{z}(t)=\int_0^t \mathbf{T}(t-s)d\mathbb{W}_s.
	\end{align*}
	We now turn to properties of the nonlinearity in \eqref{wave_map_system}.  
Let $\mathbf{n}$ and $n_0$ be as defined in \eqref{nonlinear_operator_physical}.  
Then $n_0$ satisfies the following estimate on $\dot{H}^{s-1,k}$.
	\begin{lemma}\label{lemma_regularity_nonlinear}
		Let $n\geq 5$ and $s,k>0$ satisfy \begin{align*}
			\frac{n}{2}-1<s<\frac{n}{2}-1+\frac{1}{2n-4},\ k>n,\ k\in \N.\end{align*}   
		Then we have that
		\begin{align*}
			\norm{n_0(u)-n_0(v)}_{\dot{H}^{s-1,k}_{rad}} \lesssim \norm{u-v}_{\dot{H}^{s,k}_{rad}}\left(\norm{u}_{\dot{H}^{s,k}_{rad}}+\norm{v}_{\dot{H}^{s,k}_{rad}}\right)^2\left(1+\norm{u}_{\dot{H}^{s,k}_{rad}}^{2k}+\norm{v}_{\dot{H}^{s,k}_{rad}}^{2k}\right),
		\end{align*}
    for all $u,v\in \dot{H}^{s,k}$.
	\end{lemma}
	\begin{proof}
		Let us first write the nonlinearity $n_0$ as
        \begin{align*}
		    n_0(h)=\frac{n-3}{2|x|^3}\gamma(|x|h),
		\end{align*}
        where $\gamma$ is defined in \eqref{nonlinear_perturbed_operators_SS}. 
		Since $\gamma'(0)=\gamma''(0)=0$, repeated application of the fundamental theorem of calculus yields
\begin{align*}
			&\gamma(v)-\gamma(u)\\ &=(v-u)\int_0^1  \gamma'((1-\lambda)u+\lambda v) d\lambda\\ & =  (v-u)\int_0^1 \left((1-\lambda) u+\lambda v\right) \int_0^1 \gamma''(\left((1-\lambda)u+\lambda v\right)\alpha) d\alpha d\lambda\\ & = (v-u)\int_0^1 \left((1-\lambda) u+\lambda v \right)\int_0^1 \left(\left((1-\lambda)u+\lambda v \right)\alpha\right) \int_0^1 \gamma'''(\left((1-\lambda)u+\lambda v \right)\alpha l)dl d\alpha d\lambda.
		\end{align*}
		The desired estimate then follows from  the Schauder estimate \cite[Proposition~A.1]{glogic2025globally} with $F(h) := \gamma'''(h)= 8\cos(2h)$.
	\end{proof}
From \autoref{lemma_regularity_nonlinear}, one immediately obtains an analogous estimate for $\mathbf{n}$ on $\HH_{\mathrm{rad}}^{s,k}$. In particular, if $\mathbf{u} \in \HH_{\mathrm{rad}}^{s,k}$, then $\mathbf{n}(\mathbf{u})$ is well defined as an element of $\HH_{\mathrm{rad}}^{s,k}$. With these definitions in place, we recall for the reader’s convenience the notion of a mild solution to \eqref{wave_map_system}, referring to \cite{da2014stochastic} and \cite{agresti2025nonlinear} for further details.

\begin{definition}\label{mild_solution}
Let $\mathbf{u}_0 = (u_0, \hat{u}_0)$ be an $\mathcal{F}_0$-measurable random variable with values in $\HH_{\mathrm{rad}}^{s,k}$, and let $\tau$ be a stopping time.  
A progressively measurable process $\mathbf{u} : \Omega \times [0, +\infty) \to \HH_{\mathrm{rad}}^{s,k}$ is called a \emph{mild solution} of \eqref{wave_map_system} on $[0,\tau]$ if  
\(\mathbf{u} \in C([0, \tau]; \HH_{\mathrm{rad}}^{s,k})\) $\mathbb{P}$-a.s. and satisfies
\begin{align*}
\mathbf{u}(t)
  = \mathbf{T}(t)\mathbf{u}_0
  + \int_0^t \mathbf{T}(t-s)\mathbf{n}(\mathbf{u}(s))\,ds
  + \int_0^t \mathbf{T}(t-s)\,d\mathbb{W}_s,
  \qquad \mathbb{P}\text{-a.s. for all } t \in [0,\tau].
\end{align*}
Furthermore, a progressively measurable process $\boldsymbol{\mathcal{U}} = (\mathcal{U}, \hat{\mathcal{U}}) : \Omega \times [0, +\infty) \to \HH^{s,k}$ is called a \emph{corotational mild solution} of \eqref{eq_intro_noise} on $[0,\tau]$ if it is induced by a mild solution $\mathbf{u} = (u, \hat{u})$ of \eqref{wave_map_system} through
\[
\mathcal{U}^i(t,X) = u(t, |X|)X^i, \qquad
\hat{\mathcal{U}}^i(t,X) = \hat{u}(t, |X|)X^i.
\]
\end{definition}

We are now in a position to state the local well-posedness result for \eqref{wave_map_system}.
	\begin{proof}[Proof of \autoref{Thm_well_posed}.]
		We employ a fixed point argument and look for solutions of\begin{align*}
			\mathbf{u}(t)=\mathbf{T}(t)\mathbf{u}_0+\int_0^t \mathbf{T}(t-s)\mathbf{n(\mathbf{u}(s))} ds+\int_0^t \mathbf{T}(t-s)d\mathbb{W}_s.
		\end{align*}
		Due to the regularity of $\mathcal{W}_t,$ we define 
        \begin{align*}
        \mathbf{z}(t)=\int_0^t \mathbf{T}(t-s)d\mathbb{W}_s\in C([0,+\infty);\HH_{rad}^{s,k})\quad \mathbb{P}\text{-a.s.}    
        \end{align*}
        Then, finding a solution $\mathbf{u}$ is equivalent to solving the fixed point problem
		\begin{align}\label{fixed_point_problem_v}
			\mathbf{w}(t)=\mathbf{T}(t)\mathbf{u}_0+\int_0^t \mathbf{T}(t-s)\mathbf{n(\mathbf{w}(s)+\mathbf{z}(s))} ds. 
		\end{align}
		Let $r(\omega)=\lVert \mathbf{u}_0(\omega)\rVert_{\HH_{rad}^{s,k}}<+\infty \ \mathbb{P}$-a.s. and define, for such $\omega$, the map
		\begin{align*}
			\Gamma[\mathbf{h}](t)=\mathbf{T}(t)\mathbf{u}_0+\int_0^t \mathbf{T}(t-s)\mathbf{n(\mathbf{h}(s)+\mathbf{z}(s))} ds.
		\end{align*}
		Let $M\geq 1$ be such that 
		\begin{align}\label{growth_wave_semigroup}
			\lVert \mathbf{T}(t)\rVert_{\HH_{rad}^{s,k}}\leq M e^t.
		\end{align}
		We aim to show that $\Gamma$ has a fixed point in the closed ball of radius $1 + (M+1) r(\omega)$ in $C([0,T^*(\omega)]; \HH_{rad}^{s,k})$ for sufficiently small $T^*(\omega)$. For notational simplicity, we omit the $\omega$ in what follows. Since $\mathbf{T}(t)$ is a strongly continuous semigroup on $\HH_{rad}^{s,k}$, it follows immediately that $\Gamma[\mathbf{h}]\in C([0,+\infty);\HH_{rad}^{s,k})$ whenever $\mathbf{h}\in C([0,+\infty);\HH_{rad}^{s,k})$. Moreover, by \autoref{lemma_regularity_nonlinear}
		\begin{align*}
			\lVert\Gamma [\mathbf{h}](t)\rVert_{\HH_{rad}^{s,k}}&\leq M e^{T^*} r+C_{M,s,k}\left(\lVert \mathbf{z}\rVert_{C([0,T^*];\HH_{rad}^{s,k})}+r\right)^3(1+\lVert \mathbf{z}\rVert_{C([0,T^*];\HH_{rad}^{s,k})}^{2k}+r^{2k})(e^{T^*}-1)\\ & \leq  1+(M+1)r,
		\end{align*}
		if $T^*$ is sufficiently small. Hence, $\Gamma$ maps the closed ball of radius $1 + (M+1) r$ into itself. Next, we verify that $\Gamma$ is a contraction. Using the properties of the wave semigroup, we have
		\begin{multline*}
			\lVert\Gamma[\mathbf{h}_1](t)-\Gamma[\mathbf{h}_2](t)\rVert_{\HH_{rad}^{s,k}}\\ \lesssim \lVert \mathbf{h}_1-\mathbf{h}_2\rVert_{C([0,T^*];\HH_{rad}^{s,k})} \left(\lVert \mathbf{z}\rVert_{C([0,T^*];\HH_{rad}^{s,k})}+r\right)^2(1+\lVert \mathbf{z}\rVert_{C([0,T^*];\HH_{rad}^{s,k})}^{2k}+r^{2k})(e^{T^*}-1)  \\  \leq \frac{1}{2}\lVert\mathbf{h}_1-\mathbf{h}_2\rVert_{C([0,T^*];\HH_{rad}^{s,k})}, 
		\end{multline*}
		for possibly smaller $T^*$. This implies by standard arguments all the claims. In particular, the construction implies that $\tau^* > 0$ $\mathbb{P}$-a.s.
	\end{proof}
	
	In order to prove our main result, \autoref{main_thm}, we need some additional properties of $\mathbf{z}$, both in terms of its regularity and the support of its law. To this end, we introduce the following assumption.
	\begin{assumption}\label{hp_noise_2}
		Let $\mathcal{W}$ be an infinite-dimensional Brownian motion, adapted to $(\mathcal{F}_t)_{t\ge 0}$, with values in $H^{k+\frac{n+1}{2}}_{rad}$. Moreover, we assume that $\mathcal{W}$ is non-degenerate on $H^{k+\frac{n+1}{2}}_{rad}$.
	\end{assumption}
	By non-degeneracy, we mean that the support of $\mathcal{W}$ coincides with the full space $\mathcal{W}^{\alpha,p,k}_{0,T}$ for each choice of $\alpha,p,T$ satisfying \begin{align*}
	    T>0,\quad \alpha\in (0,\frac{1}{2}),\quad p\in (2,+\infty),\quad \alpha p>1.
	\end{align*} This property holds provided that the operator $J$ appearing in \eqref{formulation_2_noise} satisfies \begin{align*}
	    J\in L(U;H^{k+\frac{n+1}{2}}_{rad}),\quad JJ^* \text{ has finite trace and } Ker(J^*)=\{0\}.
	\end{align*} Under this assumption, which is more restrictive than \autoref{HP_noise_1}, we have  
    \begin{align*}
	    \mathbf{z}\in C([0,+\infty);\mathbf{H}^{k+\frac{n+1}{2}}).
	\end{align*}
	The following lemma, concerning the support of $\mathbf{z}$, the solution of \eqref{stochastic_wave}, is likely well-known to experts. However, since we were unable to find a suitable reference, we provide a proof for completeness.
	\begin{lemma}\label{lemma_support_z}
		Let $T>0$ and suppose that \autoref{hp_noise_2} holds. If $\mathbf{z}$ is the mild solution of \eqref{stochastic_wave} on $[0,T]$, then the support of $\mathbf{z}$ is the full space $\mathcal{E}_{0,T}^k$.
	\end{lemma}
	\begin{proof}
		Let $\tilde{\mathbf{z}} \in \mathcal{E}_{0,T}^k$ and $\varepsilon > 0$. It is enough to show that
		\begin{align*}
			\mathbb{P}\left(\lVert \mathbf{z}-\Tilde{\mathbf{z}}\rVert_{C([0,T];\mathbf{H}^{k+\frac{n+1}{2}})}\leq \eps\right)>0.
		\end{align*}
		Let $\mathbf{z}^* = (z^*, \hat{z}^*) \in \mathcal{E}_{0,T}^k$ be smooth and satisfy\footnote{This can be done introducing two families of mollifiers: one in space, $\chi_{\delta_1}(x)=\frac{1}{\delta_1^n}\chi(\frac{x}{\delta_1}),$ with $ \chi\in C^{\infty}_{c,rad}(\R^n),\ \int_{\R^n}\chi=1$, and one in time, $\hat{\chi}_{\delta_2}(s)=\frac{1}{\delta_2}\hat\chi(\frac{s}{\delta_2}),$ with $\hat\chi\in C^{\infty}_c(0,1),\ \int_0^1 \hat{\chi}(s)ds=1.$ Choosing $\delta_1, \delta_2$ sufficiently small, the couple $\left(\int_0^{\infty}\int_{\R^n} \Tilde{z}(t-s,x-y)\chi_{\delta_1}(y)\hat{\chi}_{\delta_2}(s)ds dy;\int_0^{\infty}\int_{\R^n} \hat{\Tilde{z}}(t-s,x-y)\chi_{\delta_1}(y)\hat{\chi}_{\delta_2}(s)ds dy\right)$ satisfies the required properties.}
        \begin{align*}
			\lVert {\mathbf{z}}^*-\Tilde{\mathbf{z}}\rVert_{C([0,T];\mathbf{H}^{k+\frac{n+1}{2}})}\leq \frac{\eps}{2}.    
		\end{align*}
        It is then enough to prove
		\begin{align*}
			\mathbb{P}\left(\lVert \mathbf{z}-{\mathbf{z}}^*\rVert_{C([0,T];\mathbf{H}^{k+\frac{n+1}{2}})}\leq \frac{\eps}{2}\right)>0.
		\end{align*}
		Since $\mathbf{z}^*$ is smooth, there exists $f^* \in \mathcal{W}_{0,T}^{\alpha,p,k}$ such that \begin{align*}
			{\mathbf{z}}^*(t)=\mathbf{F}(t)+\int_0^t \mathbf{T}(t-s)\mathbf{A}\mathbf{F}^*(s)ds,\ \text{with } \mathbf{F}^*(t)=(0,f^*(t)),
		\end{align*} 
		or equivalently
		\begin{align*}
			\begin{cases}
				\partial_t z^*=\hat{z}^*, \\
				\partial_t \hat{z}^*= \Delta z^*+\partial_t f^*,\\
				z^*(0)=0,\\
				\hat{z}^*(0)=0.
			\end{cases}
		\end{align*}
		Moreover, the solution map is continuous from $\mathcal{W}_{0,T}^{\alpha,p,k}$ to $\mathcal{E}^k_{0,T}$.
		Therefore, there exists a neighborhood of $f^*$ in $\mathcal{W}^{\alpha,p,k}_{0,T}$, denoted by $B_{\eps,f^*}$, such that for any ${f}^{**}\in B_{\eps,f^*} $ the solution of the wave equation with forcing $\partial_t\hat f^{**}$, which we denote by $\mathbf{z}^{**}$, satisfies
		\begin{align*}
			\lVert \mathbf{z}^{**}-{\mathbf{z}}^*\rVert_{C([0,T];\mathbf{H}^{k+\frac{n+1}{2}})}\leq \frac{\eps}{2}.    
		\end{align*}
		In conclusion, we have
		\begin{align*}
			\mathbb{P}\left(\lVert \mathbf{z}-\Tilde{\mathbf{z}}\rVert_{C([0,T];\mathbf{H}^{k+\frac{n+1}{2}})}\leq \eps\right)\geq \mathbb{P}\left(\mathcal{W}\in B_{\eps,f^*} \right). 
		\end{align*}
		The latter is positive, due to the properties of the support of the Brownian motion $\mathcal{W}$.
	\end{proof}
\subsection{Analysis in self-similar variables:~linearized dynamics}\label{prel_self_similar}
	As explained in \autoref{sec_strategy}, to study the stability of self-similar solutions, it is convenient to pass to similarity variables, thereby  reformulating the wave maps equation in these coordinates. Additionally, in this section we recall some properties of the wave operator in these variables, referring to \cite[Sections 4–7]{glogic2025globally} for detailed proofs. 
    
    For $\tilde{T}>0$ we define the new time and space variables
	\begin{align*}
		\tau=\log\left(\frac{\tilde{T}}{\tilde{T}-t}\right),\ \xi=\frac{x}{\tilde{T}-t},
	\end{align*}
	and, relative to \eqref{intro_eq_2}, let
	\begin{align*}
		U(\tau,\xi):=(\tilde{T}-t)u(t,x),\quad \hat{U}(\tau,\xi):=(\tilde{T}-t)^2 \hat{u}(t,x).
	\end{align*}
	The function $\mathbf{U}=(U,\hat{U})$ then formally satisfies the evolution equation
	\begin{align}\label{wave_map_selfsimilar_variables}
		\begin{cases}
			\partial_\tau \mathbf{U}=\mathbf{L}_0 \mathbf{U}+\mathbf{n}(\mathbf{U})\\
			\mathbf{U}(0)=\mathbf{U}_0,
		\end{cases}
	\end{align}
	where $\mathbf{L}_0 $ is the wave operator in similarity variables, as defined in \eqref{linear_operator_SS}.
	According to \cite[Proposition 4.1]{glogic2025globally}, the operator $\mathbf{L}_0$, originally defined on smooth functions, is closable on $\HH_{rad}^{s,k}$, and its closure generates a strongly continuous semigroup $\mathbf{S}_0(\tau)$ on $\HH_{rad}^{s,k}$. 
    We now decompose the solution in self-similar variables as
    \begin{align*}
        \mathbf{U}=\mathbf{\Phi}+\mathbf{\Psi},
    \end{align*}
     where $\mathbf{\Phi}$ is the static self-similar profile associated with the nonlinear wave map introduced in \eqref{vector_self_similar_blow_up}. By construction, $\mathbf{\Phi}$ generates a blowup solution of \eqref{intro_eq_2} under the self-similar rescaling. The perturbation $\mathbf{\Psi}$ then satisfies
	\begin{align*}
		\begin{cases}
			\partial_\tau \mathbf{\Psi}=\mathbf{L} \mathbf{\Psi}+\mathbf{N}(\mathbf{\Psi})\\
			\mathbf{\Psi}(0)=\mathbf{U}_0-\mathbf{\Phi},
		\end{cases}
	\end{align*}
	where $\mathbf{L}$ and $\mathbf{N}$ are defined in \eqref{linear_perturbed_operators_SS} and \eqref{nonlinear_perturbed_operators_SS}, respectively.
	As shown in \cite[Sections 5–6]{glogic2025globally}, the operator $\mathbf{L}$, like $\mathbf{L}_0$, is closable on $\HH_{rad}^{s,k}$, and its closure generates a strongly continuous semigroup $\mathbf{S}(\tau)$ on $\HH_{rad}^{s,k}$. Most importantly, there exists $\overline{\omega}\in (0, s+1-\frac{n}{2})$ such that 
	\begin{align}\label{spectrum_L}
		\{\lambda \in \sigma(\mathbf{L}):\Re \lambda\geq -\overline{\omega}\}=\{1\}.
	\end{align}
	Moreover, eigenvalue $\la=1$ is simple and it does not correspond to a genuine instability as it arises from the time-translation symmetry of \eqref{SS_sol}. The corresponding eigenfunction can thereby be explicitly computed, and is given by
	\begin{align*}
		\mathbf{g}:=(g,\xi\cdot\nabla g+2g),\quad g(\xi)=\frac{1}{\lvert \xi\rvert^2+n-4}.
	\end{align*}
	Consequently, due to the underlying spectral mapping relation between $\mathbf{S}(\tau)$ and its generator $\mathbf{L}$, we have that
    \begin{align}\label{decay_semigroup}
		\norm{\mathbf{S}(\tau)(I-P)\mathbf{K}}_{\HH_{rad}^{s,k}}\lesssim e^{-\overline{\omega}\tau} \norm{(I-P)\mathbf{K}}_{\HH_{rad}^{s,k}},
	\end{align}
    where $P$ stands for the Riesz projection of $\mathbf{L}$ relative to $\la=1$.
	Finally, according to \cite[Lemma 7.1 and Proposition A.1]{glogic2025globally}, the nonlinearity $\mathbf{N}$ is locally Lipschitz continuous in $\HH_{rad}^{s,k}$, as it satisfies the estimate
	\begin{align}\label{estimate_nonlinearity_selfsimilar}
		\norm{\mathbf{N}(\mathbf{\Psi}_1)-\mathbf{N}(\mathbf{\Psi}_2)}_{\HH_{rad}^{s,k}}&\lesssim \norm{\mathbf{\Psi}_1-\mathbf{\Psi}_2}_{\HH_{rad}^{s,k}}\left(\norm{\mathbf{\Psi}_1}_{\HH_{rad}^{s,k}}+\norm{\mathbf{\Psi}_2}_{\HH_{rad}^{s,k}}\right)\\ & \quad \times\left(\norm{\mathbf{\Phi}}_{\HH_{rad}^{s,k}}+\norm{\mathbf{\Psi}_1}_{\HH_{rad}^{s,k}}+\norm{\mathbf{\Psi}_2}_{\HH_{rad}^{s,k}}\right)\notag\\ &\quad  \times \left(1+\norm{\mathbf{\Phi}}_{\HH_{rad}^{s,k}}^{2k}+\norm{\mathbf{\Psi}_1}_{\HH_{rad}^{s,k}}^{2k}+\norm{\mathbf{\Psi}_2}_{\HH_{rad}^{s,k}}^{2k}\right).\notag
	\end{align}
    for all $\mathbf{\Psi}_1,\mathbf{\Psi}_2 \in \HH_{rad}^{s,k}$.
	\section{Stable self-similar blowup profile for stochastic corotational  wave maps}\label{sec_self_similar_analysis}
	\noindent The main goal of this section is to establish the following perturbative blowup result.
	\begin{theorem}\label{stability_lemma}
		Let $d\geq 3$ and $s,k>0$ such that \begin{align*}
			\frac{d}{2}<s<\frac{d}{2}+\frac{1}{2d},\ k>d+2,\ k\in \N.\end{align*}
		Given \autoref{hp_noise_2}, for each $T>0$ there exists $\eps>0$ small enough such that for each $\mathcal{F}_0$-measurable corotational random variable $\boldsymbol{\mathcal{U}}_0=(\mathcal{U}_0,\hat{\mathcal{U}}_0)\in \HH^{s,k}$ for which
		\begin{align*}
			\mathbb{P}\left(\norm{\boldsymbol{\mathcal{U}}_0-\boldsymbol{\mathcal{U}}_{T}(0,\cdot)}_{\HH^{s,k}}<\eps\right)>0,   
		\end{align*} there exists a set $\mathcal{N} \subseteq \Omega$ of positive probability, such that for all $\omega \in \mathcal{N}$ the unique corotational mild solution $\mathcal{U}$ of \eqref{eq_intro_noise} given by \autoref{corollary_well_posed}, satisfies $0 < \tau^{*}(\boldsymbol{\mathcal{U}}_0,\omega) < 2T$, and, moreover, we have the following decomposition
		\begin{align}\label{representation_formula_calU}
			\mathcal{U}(t,X)1_{\mathcal{N}}&=\frac{X}{\tau^{*}(\boldsymbol{\mathcal{U}}_0,\omega)-t}\Phi\left(\frac{\lvert X\rvert }{\tau^{*}(\boldsymbol{\mathcal{U}}_0,\omega)-t}\right)1_{\mathcal{N}}\notag \\ +&\frac{X}{\tau^{*}(\boldsymbol{\mathcal{U}}_0,\omega)-t}\Psi\left(\log\left(\frac{\tau^{*}(\boldsymbol{\mathcal{U}}_0,\omega)}{\tau^{*}(\boldsymbol{\mathcal{U}}_0,\omega)-t}\right),\frac{\lvert X\rvert}{\tau^{*}(\boldsymbol{\mathcal{U}}_0,\omega)-t}\right)1_{\mathcal{N}} +Xz(t,\lvert X\rvert) 1_{\mathcal{N}}
		\end{align}
		for  $t\in [0,\tau^{*}(\boldsymbol{\mathcal{U}}_0,\omega))$, where $z(t,|X|) $ is the solution at time $t$ of the wave equation \eqref{linear_wave_1} with additive noise $\mathcal{W}$ starting from zero initial data, and $\Psi$ satisfies on $\mathcal{N}$
		\begin{align*}
			\norm{\Psi(\tau)}_{\HH_{rad}^{s,k}}\rightarrow 0\quad \text{as} \quad \tau\rightarrow+\infty. 
		\end{align*}
		In particular, on $\mathcal{N}$ it holds 
		\begin{align*}
			\mathcal{U}(t,(\tau^{*}(\boldsymbol{\mathcal{U}}_0,\omega)-t)\cdot) \rightarrow (\cdot)\Phi(|\cdot|)
		\end{align*}
		locally uniformly on $\mathbb{R}^d$. As a consequence, with positive probability, the solution $\mathcal{U}(t)$ blows up in a self-similar way before $T$.

	\end{theorem}
	\subsection{Change of Variables}\label{sec_change of variables}
	Let $\mathbf{z}$ be the stochastic convolution starting from $0$, as defined in \autoref{prel_physical}, and let $\mathbf{\Phi}$ be the background self-similar profiles defined in \eqref{vector_self_similar_blow_up}. As discussed in \autoref{sec_strategy}, for each $0<\tilde{T}\leq 2T$, we introduce the change of variables 
	\begin{align*}
		\tau=\log\left(\frac{\tilde{T}}{\tilde{T}-t}\right),\ \xi=\frac{x}{\tilde{T}-t},
	\end{align*}
	and define 
	\begin{align}\label{change_of_variables}
		\phi(x,t)&:=\frac{1}{\tilde{T}-t}\Phi\left(\frac{\lvert x\rvert}{\tilde{T}-t}\right)=\frac{2}{\lvert x\rvert}\operatorname{arctan}\left(\frac{\lvert x\rvert}{\sqrt{n-4}(\tilde{T}-t)}\right),\notag\\
		\hat{\phi}(x,t)&:=\partial_t \phi(x,t)=\frac{1}{(\tilde{T}-t)^2}\hat{\Phi}\left(\frac{\lvert x\rvert}{\tilde{T}-t}\right)\notag\\
		Z(\tau, \xi)&:=(\tilde{T}-t)z(t,x),\notag\\
		\hat{Z}(\tau, \xi)&:=(\tilde{T}-t)^2\hat{z}(t,x),\notag\\
		\mathbf{Z}(\tau,\xi)&:=(Z(\tau, \xi),\hat{Z}(\tau, \xi)),
	\end{align}    
	Writing $u = \phi + z + \psi$ with $\hat{\psi} = \partial_t \psi$, we define
	\begin{align*}
		\Psi(\tau,\xi):=(\tilde{T}-t)\psi(t,x),\quad \hat{\Psi}(\tau,\xi):=(\tilde{T}-t)^2\hat{\psi},\quad (t,x), \mathbf{\Psi}=(\Psi,\hat{\Psi}).
	\end{align*}
	Then $\mathbf{\Psi}$ satisfies \eqref{stability_system_ss} with initial data \eqref{initial_conditions_self_sim_perturb}.
	Thus, to prove \autoref{stability_lemma}, it suffices to study the system \eqref{stability_system_ss} pathwise. The existence of a positive-probability set follows from \autoref{lemma_support_z}, which ensures that $\mathbf{z}$ has the required properties, together with the fact that $\mathbf{z}$ is independent of $\mathcal{F}_0$.
    
	\subsubsection{Analysis of the perturbed system}
	Let us define 
	\begin{align*}
		\mathcal{X}:=\left\{f\in C([0,+\infty);\HH_{rad}^{s,k}):\norm{\phi}_{\mathcal{X}}=\sup_{t\in [0,+\infty)}e^{\overline{\omega} t}\norm{\phi(t)}_{\HH_{rad}^{s,k}}\right\},
	\end{align*}
    where $\overline{\omega}$ is from the previous section; see \eqref{spectrum_L}.
    %
    %
    Let $\mathcal{X}_{r}$ (resp $B_{r},\ C_{T}B_{r}$) denote the closed ball of radius $r$ in $\mathcal{X}$ (resp. $\HH_{rad}^{s,k},\ C([0,2T];\mathbf{H}^{k+\frac{n+1}{2}})$) and $P$ be the Riesz projection of $\mathbf{L}$ associated to the eigenvalue $1$.
	We study \eqref{stability_system_ss} for a generic initial condition $\mathbf{v}\in \HH_{rad}^{s,k}$ using a Lyapunov–Perron type argument.
	For $\mathbf{v}\in \HH_{rad}^{s,k},\ \mathbf{K}\in \mathcal{X},\ \mathbf{Z}\in C([0,+\infty);\HH_{rad}^{s,k}),\ \mathbf{Z}=(Z,\hat{Z})$, define
	\begin{align*}
		C(\mathbf{v},\mathbf{K},\mathbf{Z}):=P\left(\mathbf{v}+\int_0^{+\infty} e^{-s}\left(\mathbf{N}(\mathbf{K}+\mathbf{Z})+\mathbf{V}\mathbf{Z}\right) ds\right).
	\end{align*}
	We look for fixed points of the map 
	\begin{align*}
		K_{\mathbf{v},\mathbf{z}}(\mathbf{\Psi})(\tau)&:=\mathbf{S}(\tau)(\mathbf{v}- C(\mathbf{v},\mathbf{\Psi},\mathbf{Z}))+\int_0^\tau \mathbf{S}(\tau-s)\left(\mathbf{N}(\Psi+\mathbf{Z})+\mathbf{V}\mathbf{Z}\right) ds,
	\end{align*}
	where $\mathbf{Z}$ is obtained from $\mathbf{z}$ via the change of variables \eqref{change_of_variables}.
	This is the content of the following lemma.
	\begin{lemma}\label{well-posed_system_stabilized}
		For all sufficiently small $\delta>0$ and all sufficiently large $\mathcal{C}>0$, the
		following holds. If $\mathbf{v}\in B_{{\delta}/{\mathcal{C}}}$ and $\mathbf{z}\in C_{T}B_{\delta/\mathcal{C}}$, then there exits a unique $\mathbf{\Psi}  \in \mathcal{X}_{\delta}$ such that
        \begin{align*}
            \mathbf{\Psi}=K_{\mathbf{v},\mathbf{z}}(\mathbf{\Psi}).
        \end{align*}
		Moreover, the map $(\mathbf{v},\mathbf{z})\rightarrow \mathbf{\Psi}$ is Lipschitz continuous from $B_{\delta/\mathcal{C}}\times C_{T}B_{\delta/\mathcal{C}}$ into $\mathcal{X}_{\delta}$.
	\end{lemma}
	\begin{proof}
		We start with the following observation on $\mathbf{V}\mathbf{Z}$. For each $l\in \N,\ l\geq \frac{n}{2}-3$ and for a multi-index $\alpha\in\mathbb{N}^n$ of degree $\lvert \alpha\rvert=l$, the Leibniz rule gives
		\begin{multline*}
			\partial^{\alpha}\left(\frac{1}{(\lvert\xi\rvert^2+n-4)^2}Z(\tau,\xi)\right)=\sum_{\substack{\beta\in \mathbb{N}^n:\\ \beta\leq \alpha}} \binom{\alpha}{\beta}\partial^{\beta}\left(\frac{1}{(\lvert\xi\rvert^2+n-4)^2}\right)\partial^{\alpha-\beta}Z(\tau,\xi)\\
			=\sum_{i=0}^l \sum_{\substack{\beta\in \mathbb{N}^n:\\ \lvert \beta\rvert=i}}\tilde{T}^{1+l-i} e^{-(1+l-i)\tau}\binom{\alpha}{\beta} \partial^{\beta}\left(\frac{1}{(\lvert\xi\rvert^2+n-4)^2}\right)\partial^{\alpha-\beta} z (\tilde{T}(1-e^{-\tau}),\xi \tilde{T} e^{-\tau}).
		\end{multline*}
        We now estimate the $L^2$ norm of each term. For a given $i=\lvert \beta\rvert$, we have two cases.
		If $i\leq  \frac{n-8}{2}$, choosing $\eps <\frac{1}{4}$, by H\"older's inequality with $p=\frac{n+\eps}{8+2i}$, we have
		\begin{multline*}
			\left(\int_{\R^n}\left\lvert\partial^{\beta}\left(\frac{1}{(\lvert\xi\rvert^2+n-4)^2}\right) \partial^{\alpha-\beta}z (\tilde{T}(1-e^{-\tau}),\xi \tilde{T} e^{-\tau})\right\rvert^2 d\xi\right)^{1/2}\\ \lesssim \left(\int_{\R^n}\left\lvert\nabla^{l-i}z (\tilde{T}(1-e^{-\tau}),\xi \tilde{T} e^{-\tau})\right\rvert^\frac{2(n+\eps)}{n+\eps-8-2i} d\xi\right)^{\frac{n+\eps-8-2i}{2(n+\eps)}}\\  =\frac{\norm{\nabla^{l-i}z(\tilde{T}(1-e^{-\tau}))}_{L^{\frac{2(n+\eps)}{n+\eps-8-2i}}}}{(\tilde{T} e^{-\tau})^{\frac{n(n+\eps-8-2i)}{2(n+\eps)}}}  \lesssim \frac{\norm{z(\tilde{T}(1-e^{-\tau}))}_{H^l\cap W^{l,\infty}}}{(\tilde{T} e^{-\tau})^{\frac{n(n+\eps-8-2i)}{2(n+\eps)}}}.
		\end{multline*}
		On the other hand, if $i> \frac{n-8}{2}$, we have
		\begin{align*}
			\left(\int_{\R^n}\left\lvert\partial^{\beta}\left(\frac{1}{(\lvert\xi\rvert^2+n-4)^2}\right)\partial^{\alpha-\beta}z (\tilde{T}(1-e^{-\tau}),\xi \tilde{T} e^{-\tau})\right\rvert^2 d\xi\right)^{1/2}&\lesssim \norm{\nabla^{l-i}z(\tilde{T}(1-e^{-\tau}))}_{L^\infty}\\ &\leq \norm{z(\tilde{T}(1-e^{-\tau}))}_{H^l\cap W^{l,\infty}}.
		\end{align*}
		Therefore
		\begin{align*}
			\norm{ \nabla^l\left(\frac{1}{(\lvert\xi\rvert^2+n-4)^2}Z(\tau,\xi)\right)}&\lesssim \tilde{T} e^{-\tau}\norm{z(\tilde{T}(1-e^{-\tau}))}_{H^l\cap W^{l,\infty}}\left(\max_{i=0,\dots \lfloor{\frac{n-8}{2}}\rfloor }(\tilde{T} e^{-\tau})^{l-i-\frac{n(n+\eps-8-2i)}{2(n+\eps)}} +1\right)\\ & \lesssim e^{-\tau}\norm{z(\tilde{T}(1-e^{-\tau}))}_{H^l\cap W^{l,\infty}} \left(1+e^{-\tau(l-\frac{n}{2}+4-\frac{\eps n}{2(n+\eps)})}\right)\\ & \lesssim  e^{-\tau}\norm{z(\tilde{T}(1-e^{-\tau}))}_{H^l\cap W^{l,\infty}} ,
		\end{align*}
		due to our choice of $\eps$. Consequently, by interpolation and since $k>s>\frac{n}{2}-1$ and $\overline{\omega}<s+1-\frac{n}{2}<1$, we obtain
		\begin{align}\label{estimate_linear_stochastic}
			\norm{\mathbf{V}\mathbf{Z}(\tau)}_{\HH_{rad}^{s,k}}=\norm{\frac{8(n-4)(n-3)}{(\lvert \cdot\rvert^2+n-4)^2}Z(\tau)}_{\dot{H}^{s-1}_{rad}\cap \dot{H}^{k-1}_{rad}}&\lesssim e^{-\tau}\sup_{t\in [0,2T]}\norm{ z(t)}_{H^{k+\frac{n+1}{2}}_{rad}}.
		\end{align}
		Similarly, and more simply, 
		\begin{align}\label{estimate_zeta}
			\norm{\mathbf{Z}(\tau)}_{\HH_{rad}^{s,k}}&\lesssim e^{-\hat{\omega}\tau}\sup_{t\in [0,2T]}\norm{\mathbf{z}(t)}_{\HH_{rad}^{s,k}},
		\end{align}
		where $\hat{\omega}=s+1-\frac{n}{2}<1$. Note that, by our assumptions, $\overline{\omega}<\hat{\omega}<1$.
		Having these relations in mind, we can run the fixed point argument. First, we need to show that $\mathbf{K}_{\mathbf{v},\mathbf{z}}$ maps $\mathcal{X}_{\delta}$ into itself. In view of \eqref{decay_semigroup}, we rewrite it as
		\begin{align*}
			\mathbf{K}_{\mathbf{v},\mathbf{z}}(\mathbf{\Psi})(\tau)=\ &\mathbf{S}(\tau)(I-P)\mathbf{v}+\int_0^\tau \mathbf{S}(\tau-s)(I-P)\left(\mathbf{N}(\mathbf{\Psi}+\mathbf{Z})+\mathbf{V}\mathbf{Z} \right) ds  \\ & -\int_{\tau}^{+\infty} e^{\tau-s} P\left(\mathbf{N}(\mathbf{\Psi}+\mathbf{Z})+\mathbf{V}\mathbf{Z} \right) ds.
		\end{align*}
		Thanks to our assumptions and estimates  \eqref{decay_semigroup},  \eqref{estimate_nonlinearity_selfsimilar}, \eqref{estimate_linear_stochastic} and \eqref{estimate_zeta}, we have
		\begin{align*}
			\norm{\mathbf{K}_{\mathbf{v},\mathbf{z}}(\mathbf{\Psi})(\tau)}_{\HH_{rad}^{s,k}}\lesssim \ & e^{-\overline{\omega} \tau}\frac{\delta}{\mathcal{C}}+\int_0^\tau e^{-\overline{\omega}(\tau-s)} \left(e^{-2\overline{\omega} s}\delta^2+e^{-2\hat{\omega} s}\frac{\delta^2}{\mathcal{C}^2}+e^{- s}\frac{\delta}{\mathcal{C}}\right)ds\\ & +\int_{\tau}^{+\infty} e^{\tau-s}\left(e^{-2\overline{\omega} s}\delta^2+e^{-2\hat{\omega} s}\frac{\delta^2}{\mathcal{C}^2}+e^{- s}\frac{\delta}{\mathcal{C}}\right)ds\\  \lesssim \ & e^{-\overline{\omega} \tau}\left(\delta^2+\frac{\delta}{\mathcal{C}}\right),
		\end{align*}
		since $\overline{\omega}<\hat{\omega}<1$.
		The last estimate implies the claim provided $\delta$ is sufficiently small and $\mathcal{C}$ is sufficiently large. Secondly, we need to show that $\mathbf{K}_{\mathbf{v},\mathbf{z}}$ is a contraction. This follows from computations analogous to those above. Indeed, using \eqref{estimate_zeta}, \eqref{decay_semigroup}, and \eqref{estimate_nonlinearity_selfsimilar}, we have
		\begin{align*}
			\| \mathbf{K}_{\mathbf{v},\mathbf{z}}(\mathbf{\Psi}_1)&(\tau)-\mathbf{K}_{\mathbf{v},\mathbf{z}}(\mathbf{\Psi}_2)(\tau)\|_{\HH_{rad}^{s,k}}\\  \lesssim \ & \int_0^{\tau} e^{-\overline{\omega}(\tau-s)}\left(\norm{\mathbf{Z}(s)}_{\HH_{rad}^{s,k}}+\norm{\mathbf{\Psi}_1(s)}_{\HH_{rad}^{s,k}}+\norm{\mathbf{\Psi}_2(s)}_{\HH_{rad}^{s,k}}\right)\norm{\mathbf{\Psi}_1(s)-\mathbf{\Psi}_2(s)}_{\HH_{rad}^{s,k}}ds\\ & +\int_{\tau}^{+\infty} e^{\tau-s} \left(\norm{\mathbf{Z}(s)}_{\HH_{rad}^{s,k}}+\norm{\mathbf{\Psi}_1(s)}_{\HH_{rad}^{s,k}}+\norm{\mathbf{\Psi}_2(s)}_{\HH_{rad}^{s,k}}\right)\norm{\mathbf{\Psi}_1(s)-\mathbf{\Psi}_2(s)}_{\HH_{rad}^{s,k}} ds\\  \lesssim \ &e^{-\overline{\omega} \tau}(\delta+\frac{\delta}{\mathcal{C}})\norm{\mathbf{\Psi}_1-\mathbf{\Psi}_2}_{\mathcal{X}}.
		\end{align*}
		Hence by choosing $\delta$ sufficiently small and $\mathcal{C}$ sufficiently large, we obtain the contraction estimate
		\begin{align}\label{contraction_estimate}
			\norm{\mathbf{K}_{\mathbf{v},\mathbf{z}}(\mathbf{\Psi}_1)-\mathbf{K}_{\mathbf{v},\mathbf{z}}(\mathbf{\Psi}_2)}_{\mathcal{X}}\leq \frac{1}{4}   \norm{\mathbf{\Psi}_1-\mathbf{\Psi}_2}_{\mathcal{X}}.
		\end{align}
		It remains to show the continuity of the solution map. Let $\mathbf{\Psi}_1$ corresponds to $\mathbf{v}_1,\mathbf{z}_1$ and $\mathbf{\Psi}_2$ corresponds to $\mathbf{v}_2,\mathbf{z}_2$. Using the triangle inequality, \eqref{decay_semigroup}, and \eqref{estimate_nonlinearity_selfsimilar}, we get
		\begin{align*}
			\|\mathbf{\Psi}_1(\tau)&-\mathbf{\Psi}_2(\tau)\|_{\HH_{rad}^{s,k}}= \norm{\mathbf{K}_{\mathbf{v}_1,\mathbf{z}_1}(\mathbf{\Psi}_1)(\tau)-\mathbf{K}_{\mathbf{v}_2,\mathbf{z}_2}(\mathbf{\Psi}_2)(\tau)}_{\HH_{rad}^{s,k}}\\ \leq \ & \norm{\mathbf{K}_{\mathbf{v}_1,\mathbf{z}_1}(\mathbf{\Psi}_1)(\tau)-\mathbf{K}_{\mathbf{v}_2,\mathbf{z}_1}(\mathbf{\Psi}_1)(\tau)}_{\HH_{rad}^{s,k}}\\ & +\norm{\mathbf{K}_{\mathbf{v}_2,\mathbf{z}_2}(\mathbf{\Psi}_1)(\tau)-\mathbf{K}_{\mathbf{v}_2,\mathbf{z}_1}(\mathbf{\Psi}_1)(\tau)}_{\HH_{rad}^{s,k}}\\ &+\norm{\mathbf{K}_{\mathbf{v}_2,\mathbf{z}_2}(\mathbf{\Psi}_2)(\tau)-\mathbf{K}_{\mathbf{v}_2,\mathbf{z}_2}(\mathbf{\Psi}_1)(\tau)}_{\HH_{rad}^{s,k}}\\  \lesssim \ & e^{-\overline{\omega} \tau}\norm{\mathbf{v}_1-\mathbf{v}_2}_{\HH_{rad}^{s,k}}\\ &+ e^{-\overline{\omega} \tau}\left(\norm{\mathbf{z}_1}_{C([0,2T];\mathbf{H}^{k+\frac{n+1}{2}})}+\norm{\mathbf{z}_2}_{C([0,2T];\mathbf{H}^{k+\frac{n+1}{2}})}+\delta\right) \norm{\mathbf{z}_1-\mathbf{z}_2}_{C([0,2T];\mathbf{H}^{k+\frac{n+1}{2}})}\\ & +e^{-\overline{\omega} \tau} \norm{\mathbf{z}_1-\mathbf{z}_2}_{C([0,2T];\mathbf{H}^{k+\frac{n+1}{2}})}\\ &+\norm{\mathbf{K}_{\mathbf{v}_2,\mathbf{z}_2}(\mathbf{\Psi}_2)(\tau)-\mathbf{K}_{\mathbf{v}_2,\mathbf{z}_2}(\mathbf{\Psi}_1)(\tau)}_{\HH_{rad}^{s,k}}.
		\end{align*}
		Therefore, applying \eqref{contraction_estimate}, we deduce
		\begin{align*}
			\norm{\mathbf{\Psi}_1(\tau)-\mathbf{\Psi}_2(\tau)}_{\mathcal{X}}\lesssim \ & \norm{\mathbf{v}_1-\mathbf{v}_2}_{\HH_{rad}^{s,k}}\\ &+ \left(\norm{\mathbf{z}_1}_{C([0,2T];\mathbf{H}^{k+\frac{n+1}{2}})}+\norm{\mathbf{z}_2}_{C([0,2T];\mathbf{H}^{k+\frac{n+1}{2}})}+\delta\right) \norm{\mathbf{z}_1-\mathbf{z}_2}_{C([0,2T];\mathbf{H}^{k+\frac{n+1}{2}})}\\ & +\norm{\mathbf{z}_1-\mathbf{z}_2}_{C([0,2T];\mathbf{H}^{k+\frac{n+1}{2}})}
		\end{align*}
		and the claim follows.
	\end{proof}
    Note that, due to our construction, the initial perturbation $\mathbf{\Psi}_{0,\tilde{T}}$
  and the stochastic term $\mathbf{Z}$ both depend on $\tilde{T}$
 . Explicitly, we have
	\begin{align}
\mathbf{\Psi}_{0,\tilde{T}}:=U(\tilde{T}, u_0, \hat{u}_0)&=\begin{bmatrix}
			\tilde{T} u_0(\tilde{T} \cdot)-{T} u_0({T} \cdot)\\
			\tilde{T}^2 \hat{u}_0(\tilde{T} \cdot)-{T}^2 \hat{u}_0({T} \cdot)
		\end{bmatrix}+\begin{bmatrix}
			{T} u_0({T} \cdot)-\Phi\\
			{T}^2 \hat{u}_0({T} \cdot)-\hat{\Phi}
		\end{bmatrix}\label{initial_cond_form_1}\\ & =\begin{bmatrix}
			\frac{\tilde{T}}{T} \Psi_{0,T}(\frac{\tilde{T}}{T} \cdot)\\
			\frac{\tilde{T}^2}{T^2} \hat{\Psi}_{0,T}(\frac{\tilde{T}}{T} \cdot)
		\end{bmatrix}+\begin{bmatrix}
			\frac{\tilde{T}}{T} \Phi(\frac{\tilde{T}}{T} \cdot)-\Phi\\
			\frac{\tilde{T}^2}{T^2} \hat{\Phi}(\frac{\tilde{T}}{T} \cdot)-\hat{\Phi}
		\end{bmatrix}.\label{initial_cond_form_2}
	\end{align}
    Similarly, the stochastic term in self-similar coordinates is
    \begin{align*}
		\mathbf{Z}(\tau,\xi)&:=\mathbf{Z}_{\tilde{T}}(\tau, \xi)=(Z_{\tilde{T}},\hat{Z}_{\tilde{T}}), \text{ where }\\ \
		Z(\tau, \xi)&:=Z_{\tilde{T}}(\tau, \xi)=\tilde{T} e^{-\tau}z(\tilde{T}(1-e^{-\tau}), \xi \tilde{T} e^{-\tau}),\\
		\hat{Z}(\tau, \xi)&:=\hat{Z}_{\tilde{T}}(\tau, \xi)=\tilde{T}^2 e^{-2\tau}\hat{z}(\tilde{T}(1-e^{-\tau}), \xi \tilde{T} e^{-\tau}).
	\end{align*}
	Both expressions \eqref{initial_cond_form_1} and \eqref{initial_cond_form_2} will be used in the following analysis.
	
	The result of \autoref{well-posed_system_stabilized} applies to our initial condition $\mathbf{\Psi}_{0,\tilde{T}}$, provided that $\tilde{T}$ is sufficiently close to $T$ and that the perturbation of the initial data is small in the sense that $\norm{\mathbf{u}_0-\frac{1}{T}\mathbf{\Phi}\left(\frac{\cdot}{T}\right)}_{\HH_{rad}^{s,k}}$ is sufficiently small.
    The following lemma makes this statement precise and quantifies the dependence of $\mathbf{\Psi}_{0,\tilde{T}}$ on $\tilde{T}.$
    \begin{lemma}\label{estimate_initial_cond}
		For each $\delta\in (0,\frac{T}{2}]$ and $(u_0,\hat{u}_0)\in \HH_{rad}^{s,k}$ the map \begin{align*}
			\tilde{T}\mapsto U(\tilde{T}, u_0, \hat{u}_0): [T-\delta,T+\delta]\rightarrow \HH_{rad}^{s,k}
		\end{align*}
		is continuous. Moreover,
		\begin{align*}
			\lVert U(\tilde{T}, u_0, \hat{u}_0)\rVert_{\HH_{rad}^{s,k}}\lesssim \lvert \tilde{T}-T\rvert+\norm{\mathbf{u}_0-\frac{1}{T}\mathbf{\Phi}\left(\frac{\cdot}{T}\right)}_{\HH_{rad}^{s,k}}\quad \text{for all} \quad \tilde{T}\in \left[\frac{1}{2}T,\frac{3}{2}T\right].
		\end{align*}
	\end{lemma}
	\begin{proof}
		Both statements follow from \cite[Lemma 8.2]{glogic2025globally}, by observing that \eqref{initial_cond_form_2} corresponds to \cite[equation (8.7)]{glogic2025globally}, up to renaming $\frac{\tilde{T}}{T}$ as $T$.
	\end{proof}
	Thanks to \autoref{estimate_initial_cond}, there exists $\delta>0$ sufficiently small and $N>0$ large enough such that $\frac{\delta}{N}<\frac{T}{2}$ and, for each $\tilde{T}\in [T-\frac{\delta}{N}, T+\frac{\delta}{N}]$ and $\norm{\mathbf{u}_0-\frac{1}{T}\mathbf{\Phi}\left(\frac{\cdot}{T}\right)}_{\HH_{rad}^{s,k}}<\frac{\delta}{N^2}$, we have 
	\begin{align*}
		\lVert U(\tilde{T}, u_0, \hat{u}_0)\rVert_{\HH_{rad}^{s,k}}\leq \frac{\delta}{\mathcal{C}},
	\end{align*}
	for $\mathcal{C}$ sufficiently large, as required in \autoref{well-posed_system_stabilized}. Consequently, there exists a unique $\mathbf{\Psi}_{\tilde{T}}$ solving
	\begin{align*}
		\mathbf{\Psi}_{\tilde{T}}(\tau)=\ &\mathbf{S}(\tau)\left(U(\tilde{T}, u_0, \hat{u}_0)-C(U(\tilde{T}, u_0, \hat{u}_0),\mathbf{\Psi}_{\tilde{T}},\mathbf{Z}_{\tilde{T}})\right)  
		\\ &+\int_0^\tau \mathbf{S}(\tau-s)\left(\mathbf{N}(\mathbf{\Psi}_{\tilde{T}}+\mathbf{Z}_{\tilde{T}})+\mathbf{V}\mathbf{Z}_{\tilde{T}}\right) ds
	\end{align*}
	It remains to show that there exists a choice of $\tilde{T}$ such that \begin{align*}
	    C(U(\tilde{T}, u_0, \hat{u}_0),\mathbf{\Psi}_{\tilde{T}},Z_{\tilde{T}})=0.
	\end{align*} This is the content of the following lemma.
	\begin{lemma}\label{no_corrector}
		There exist constants $\delta>0 $ sufficiently small and $N>0$ sufficiently large such that the following holds. If \begin{align*}
		    \norm{\mathbf{u}_0-\frac{1}{T}\mathbf{\Phi}\left(\frac{\cdot}{T}\right)}_{\HH_{rad}^{s,k}}<\frac{\delta}{N^2} \quad \text{and} \quad \mathbf{z}\in C_{T}B_{\delta/N^2}
		\end{align*} then there exists $\tilde{T}\in [T-\frac{\delta}{N},T+\frac{\delta}{N}]$ such that  
		\begin{align*}
			\mathbf{\Psi}_{\tilde{T}}(\tau)=\mathbf{S}(\tau)\left(U(\tilde{T}, u_0, \hat{u}_0)\right)  
			+\int_0^\tau \mathbf{S}(\tau-s)\left(\mathbf{N}(\mathbf{\Psi}_{\tilde{T}}+\mathbf{Z}_{\tilde{T}})+\mathbf{V}\mathbf{Z}_{\tilde{T}}\right) ds.
		\end{align*}
	\end{lemma}
	\begin{proof}
		By \autoref{well-posed_system_stabilized} and \autoref{estimate_initial_cond}, we can choose  $\delta,N>0$ such that $\frac{\delta}{N}\leq \frac{T}{2}$ and for each $\tilde{T}\in [T-\frac{\delta}{N},T+\frac{\delta}{N}]$ we have a unique $\mathbf{\Psi}_{\tilde{T}}$ solving
		\begin{align*}
			\mathbf{\Psi}_{\tilde{T}}(\tau)=\ &\mathbf{S}(\tau)\left(U(\tilde{T}, u_0, \hat{u}_0)-C(U(\tilde{T}, u_0, \hat{u}_0),\mathbf{\Psi}_{\tilde{T}},\mathbf{Z}_{\tilde{T}})\right)  \\ &
			+\int_0^\tau \mathbf{S}(\tau-s)\left(\mathbf{N}(\mathbf{\Psi}_{\tilde{T}}+\mathbf{Z}_{\tilde{T}})+\mathbf{V}\mathbf{Z}_{\tilde{T}}\right) ds.
		\end{align*}
		By the definition of the map $C(U(\tilde{T}, u_0, \hat{u}_0),\mathbf{\Psi}_{\tilde{T}},\mathbf{Z}_{\tilde{T}})$,  it is sufficient to show that there exists a choice of $\tilde{T}$ possibly depending on $u_0, \hat{u}_0,T, \mathbf{z}$ such that
		\begin{align*}
			\langle C(U(\tilde{T}, u_0, \hat{u}_0),\mathbf{\Psi}_{\tilde{T}},\mathbf{Z}_{\tilde{T}}),\,\mathbf{g}\rangle_{\HH_{rad}^{s,k}}=0.
		\end{align*}
		First, we observe that, arguing as in \cite[Lemma 8.2]{glogic2025globally}, we have
		\begin{align*}
			\langle PU(\tilde{T}, u_0, \hat{u}_0),\mathbf{g}\rangle_{\HH_{rad}^{s,k}}=2\sqrt{n-4}\norm{\mathbf{g}}_{\HH_{rad}^{s,k}}^2\left(\frac{\tilde{T}}{T}-1\right)+R_{1}\left(\mathbf{u}_0-\frac{1}{T}\mathbf{\Phi}\left(\frac{\cdot}{T}\right),\frac{\tilde{T}}{T}\right),
		\end{align*}
		where $R_{1}(\mathbf{u}_0-\frac{1}{T}\mathbf{\Phi}\left(\frac{\cdot}{T}\right),\frac{\tilde{T}}{T})$ is continuous in $\tilde{T}$ and satisfies \begin{align*}
			\left\lvert R_{1}\left(\mathbf{u}_0-\frac{1}{T}\mathbf{\Phi}\left(\frac{\cdot}{T}\right),\frac{\tilde{T}}{T}\right) \right\rvert\lesssim \frac{\delta}{N^2}.  
		\end{align*}
		Here we used again the argument from \cite[Lemma 8.2]{glogic2025globally}, which applies to our case by renaming $\frac{\tilde{T}}{T}$ as $T$.
		Next, set $R_2(z,\mathbf{u}_0-\mathbf{\Phi},\tilde{T})=P\int_0^{+\infty}  e^{-s}\left(\mathbf{N}(\mathbf{\Psi}_{\tilde{T}}+\mathbf{Z}_{\tilde{T}})+\mathbf{V}\mathbf{Z}_{\tilde{T}}\right) ds$. Then the problem reduces to finding $\tilde{T}$ such that \begin{align*}
			\tilde{T} =T-T\frac{R_{1}(\mathbf{u}_0-\mathbf{\Phi},\tilde{T})}{2\sqrt{n-4}\norm{\mathbf{g}}_{\HH_{rad}^{s,k}}^2}-T\frac{R_2(z,\mathbf{u}_0-\mathbf{\Phi},\tilde{T})}{2\sqrt{n-4}\norm{\mathbf{g}}_{\HH_{rad}^{s,k}}^2}.
		\end{align*}
		Defining the map
		\begin{align*}
			F(\tilde{T}):\left[T-\frac{\delta}{\mathcal{C}},T+\frac{\delta}{\mathcal{C}}\right]\rightarrow \R,\ F(\tilde{T})=T-T\frac{R_{1}(\mathbf{u}_0-\mathbf{\Phi},\tilde{T})}{2\sqrt{n-4}\norm{\mathbf{g}}_{\HH_{rad}^{s,k}}^2}-T\frac{R_2(z,\mathbf{u}_0-\mathbf{\Phi},\tilde{T})}{2\sqrt{n-4}\norm{\mathbf{g}}_{\HH_{rad}^{s,k}}^2},
		\end{align*}
		we are left to find a fixed point of $F.$
		Observe that $R_2$ is continuous in $\tilde{T}$. Indeed, for $\tilde{T}_1,\tilde{T}_2>0$, we can write \begin{align*}
			Z_{\tilde{T}_2}(\tau,\xi)&=\tilde{T}_2(1-e^{-\tau})z(\tilde{T}_2 e^{-\tau},\xi \tilde{T}_2(1-e^{-\tau}))\\ & =\frac{\tilde{T}_2}{\tilde{T}_1}\tilde{T}_1 (1-e^{-\tau}) z\left(\frac{\tilde{T}_2}{\tilde{T}_1}\tilde{T}_1 e^{-\tau},\xi \frac{\tilde{T}_2}{\tilde{T}_1}\tilde{T}_1 (1-e^{-\tau})\right)\\ & = \tilde{T}_1(1-e^{-\tau})\Tilde{z}_{\tilde{T}_1,\tilde{T}_2}\left(\tilde{T}_1 e^{-\tau},\xi \tilde{T}_1(1-e^{-\tau})\right),\quad \Tilde{z}_{\tilde{T}_1,\tilde{T}_2}(t,x)=\frac{\tilde{T}_2}{\tilde{T}_1} z\left(\frac{\tilde{T}_2}{\tilde{T}_1}t,\frac{\tilde{T}_2}{\tilde{T}_1}x\right).
		\end{align*}
		In particular $\Tilde{z}_{\tilde{T}_1,\tilde{T}_2}\rightarrow z \in C([0,2T];H^{k+\frac{n+1}{2}}_{rad})$ as $\tilde{T}_1\rightarrow \tilde{T}_2.$ A similar argument applies to the second component of $\mathbf{Z}$, and by the continuity of the solution map, \emph{cf.} \autoref{well-posed_system_stabilized}, we also have $\mathbf{\Psi}_{\tilde{T}_1}\rightarrow \mathbf{\Psi}_{\tilde{T}_2}$ in $\mathcal{X}$ as $\tilde{T}_1\rightarrow \tilde{T}_2.$ Combining this information with the estimates \eqref{estimate_linear_stochastic}, \eqref{estimate_zeta} and \eqref{estimate_nonlinearity_selfsimilar}, implies that $R_2(z,\mathbf{u}_0-\mathbf{\Phi},\tilde{T})$ is continuous in $\tilde{T}$. Moreover, by \eqref{estimate_linear_stochastic}, \eqref{estimate_zeta} and \eqref{estimate_nonlinearity_selfsimilar}, we have that
		\begin{align*}
			\lvert R_2(z,\mathbf{u}_0-\mathbf{\Phi},\tilde{T})\rvert & \lesssim \norm{\mathbf{g}}_{\HH_{rad}^{s,k}}\int_0^{+\infty} e^{-s}\left(e^{-2\overline{\omega} s}\delta^2+e^{-2\overline{\omega} s}\frac{\delta^2}{N^4}+e^{-\overline{\omega} s}\frac{\delta}{N^2}\right)ds.
		\end{align*}
		Hence, by taking $\delta>0$ sufficiently small and $N>0$ sufficiently large, we can ensure
		\begin{align*}
			\lvert R_{1}(\mathbf{u}_0-\mathbf{\Phi},\tilde{T})\rvert+\lvert R_2(z,\mathbf{u}_0-\mathbf{\Phi},\tilde{T})\rvert\leq \frac{\delta}{T N}\left(2\sqrt{n-4}\norm{\mathbf{g}}_{\HH_{rad}^{s,k}}^2\right).
		\end{align*}
		Therefore, $F$ is continuous and, by choosing $\delta>0$ sufficiently small and $N>0$ sufficiently large, it maps $[T-\frac{\delta}{N},T+\frac{\delta}{N}]$ into itself. The Brouwer fixed-point theorem then guarantees the existence of a fixed point, completing the proof.
	\end{proof}
	Lastly, we need a lemma that relates the mild solution constructed in self-similar variables to that obtained in physical variables. Indeed, we have shown that there exists a unique mild solution $\mathbf{\Psi}$ to \eqref{stability_system_ss}. Moreover, since $\mathbf{\Phi}$ is a classical solution of \eqref{wave_map_selfsimilar_variables}, it is also a mild one. Consequently, their sum is a mild solution of \eqref{DPD_self_similar}.
	The following result establishes the link between this solution and the fixed points of \eqref{fixed_point_problem_v}.
	\begin{lemma}\label{lemma_reconstruction_solution}
		If $\mathbf{W}=(W,\hat{W})$ is a (global) mild solution of \eqref{DPD_self_similar} with initial condition $\mathbf{W}_0=(\tilde{T} u_0(\tilde{T}\cdot),\tilde{T}^2 \hat{u}_0(\tilde{T}\cdot))$ then the function
		\begin{align*}
			\mathbf{w}(t,x):=\left(\frac{1}{\tilde{T}-t}W\left(\log\left(\frac{\tilde{T}}{\tilde{T}-t}\right),\frac{x}{\tilde{T}-t}\right),\frac{1}{(\tilde{T}-t)^2}\hat{W}\left(\log\left(\frac{\tilde{T}}{\tilde{T}-t}\right),\frac{x}{\tilde{T}-t}\right)\right)
		\end{align*}
		is a mild solution of \eqref{fixed_point_problem_v} on $[0,\tilde{T})$ with initial condition $\mathbf{w}_0=( u_0,\hat{u}_0)$.
	\end{lemma}
	\begin{proof}
		Let us introduce, on $\HH_{rad}^{s,k}$, the family of operators $\tilde{S}_{\tilde{T}}(t)$ for $t<\tilde{T}$ defined by 
		\begin{align}\label{def_aux_semigroup}
			[\tilde{S}_{\tilde{T}}(t)\mathbf{u}_0](x):=\begin{bmatrix}  \frac{1}{\tilde{T}-t}, & 0\\ 0, & \frac{1}{(\tilde{T}-t)^2}\end{bmatrix} 
			\left[\mathbf{S}_0\left(\frac{\tilde{T}}{\tilde{T}-t}\right)\left(\tilde{T} u_0(\tilde{T}\cdot),\tilde{T}^2 \hat{u}_0(\tilde{T} \cdot)\right)  \right]\left(\frac{x}{\tilde{T}-t}\right), 
		\end{align}
		where $ \mathbf{u}_0=(u_0,\hat{u}_0).$ It is straightforward to verify that $\tilde{S}_{\tilde{T}}$ coincides with the standard wave semigroup on $\HH_{rad}^{s,k}$ for all $t<\tilde{T}$. Indeed, by density we can assume $\mathbf{u}_0\in \mathscr{S}_{rad}\times \mathscr{S}_{rad}$. In this case, the wave semigroup yields the unique classical solution of the wave equation, namely $\mathbf{u}(t)=\mathbf{T}(t)\mathbf{u}_0.$ Under the change of variables introduced in \autoref{prel_self_similar}, we define
		\begin{align*}
			\mathbf{U}(\tau,\xi)=\begin{bmatrix}\tilde{T} e^{-\tau} ,& 0\\ 0,& \tilde{T}^2 e^{-2\tau}\end{bmatrix}\mathbf{u}(\tilde{T} (1-e^{-\tau}),\xi \tilde{T} e^{-\tau}).
		\end{align*}
		Then $\mathbf{U}$ is a classical (hence also mild) solution of the wave equation in self-similar variables, with initial condition 
		$\mathbf{U}_0=(\tilde{T} u_0(\tilde{T}\cdot),\tilde{T}^2 \hat{u}_0(\tilde{T}\cdot))$, that is $\mathbf{U}(\tau,\xi)=[\mathbf{S}_0(\tau)\mathbf{U}_0](\xi).$ Observing that equation \eqref{def_aux_semigroup} corresponds to inverting this change of variables in the mild formulation of $\mathbf{U}$, the claim follows. We are now ready to prove our result. By the definition of $\mathbf{v}$, previous step and the chance of variables $r=\log(\frac{\tilde{T}}{\tilde{T}-s})$ we obtain for $t<\tilde{T}$, 
		\begin{align}
			\mathbf{w}(t,x)&=[{S}(t)\mathbf{u}_0](x)\notag\\ &+\begin{bmatrix}  \frac{1}{\tilde{T}-t}, & 0\\ 0 & \frac{1}{(\tilde{T}-t)^2}\end{bmatrix}\left[\int_0^{\log\left(\frac{\tilde{T}}{\tilde{T}-t}\right)}\mathbf{S}_0\left(\log\left(\frac{\tilde{T}}{\tilde{T}-t}\right)-s\right)\frac{(n-3)\mathbf{f}(\lvert \cdot \rvert ({W}+Z)(s)) }{\lvert \cdot \rvert^3}ds\right]\left(\frac{x}{\tilde{T}-t}\right)\notag\\ \label{mild_formulation_self_similar_change} &=[{S}(t)\mathbf{u}_0](x)\notag\\ &+\begin{bmatrix}  \frac{1}{\tilde{T}-t}, & 0\\ 0 & \frac{1}{(\tilde{T}-t)^2}\end{bmatrix}\left[\int_0^{\log\left(\frac{\tilde{T}}{\tilde{T}-t}\right)}\mathbf{S}_0\left(\log\left(\frac{\tilde{T}-s}{\tilde{T}-t}\right)\right)\frac{(n-3)\mathbf{f}(\lvert \cdot \rvert (W+Z)(\log\frac{\tilde{T}}{\tilde{T}-s})) }{(\tilde{T}-s)\lvert \cdot \rvert^3}ds\right]\left(\frac{x}{\tilde{T}-t}\right),
		\end{align}
		where 
		\begin{align*}
			\mathbf{f}(\alpha)=(0,2\alpha-\sin (2\alpha) ).
		\end{align*}
		To identify the second term, note that by the definition of $\mathbf{W}$
		\begin{align*}
			\frac{(n-3)\mathbf{f}(\lvert \cdot \rvert (W+Z)(\log\frac{\tilde{T}}{\tilde{T}-r},\cdot)) }{(\tilde{T}-r)\lvert \cdot \rvert^3}&=\begin{bmatrix}
				\tilde{T}-r, & 0 \\ 0, & (\tilde{T}-r)^2
			\end{bmatrix}\frac{\mathbf{f}((\tilde{T}-r)\lvert \cdot\rvert (w+z)(r,(\tilde{T}-r)\cdot))}{(\tilde{T}-r)^3\lvert \cdot\rvert^3}
		\end{align*}
		and, obviously, $\mathbf{S}_0\left(\log\left(\frac{\tilde{T}-r}{\tilde{T}-t}\right)\right)=\mathbf{S}_0\left(\log\left(\frac{\tilde{T}-r}{(\tilde{T}-r)-(t-r)}\right)\right)$. Using \eqref{def_aux_semigroup} with the parameter to $\tilde{T}-r$, we deduce
		\begin{multline*}
			\left[\mathbf{S}_0\left(\log\left(\frac{\tilde{T}-r}{\tilde{T}-t}\right)\right)\frac{(n-3)\mathbf{f}(\lvert \cdot \rvert (W+Z)(\log\frac{\tilde{T}}{\tilde{T}-r})) }{(\tilde{T}-r)\lvert \cdot \rvert^3}\right](\cdot)\\ =\begin{bmatrix}
				\tilde{T}-t, & 0 \\ 0, & (\tilde{T}-t)^2
			\end{bmatrix} [\mathbf{T}(t-r) \mathbf{n}((\mathbf{w}+\mathbf{z})(r))]\left((\tilde{T}-t)\cdot\right).   
		\end{multline*}
		Substituting this expression into \eqref{mild_formulation_self_similar_change} yields the claim.
	\end{proof}
	The proof of \autoref{stability_lemma} now follows directly from \autoref{no_corrector} and \autoref{lemma_reconstruction_solution}.
	\begin{proof}[Proof of \autoref{stability_lemma}.]
		Let us first select $\delta>0$ and $N>0$ so that the assumptions of \autoref{no_corrector} are satisfied, and choose $\varepsilon $ in the statement of \autoref{stability_lemma} small enough such that
		\begin{align*}
			\mathbb{P}\left(\norm{\mathbf{u}_0-\mathbf{u}_{T}(0,\cdot)}_{\HH_{rad}^{s,k}}< \frac{\delta}{N^2}\right)>0.   
		\end{align*}
		Consequently, by \autoref{lemma_support_z} and the independence of $\mathbf{z}$ from $\mathcal{F}_0$, there exists a measurable set of positive probability, denoted by $\mathcal{N}_{\delta,N}$, such that
		\begin{align*}
			\sup_{t\in [0,2T]}\norm{\mathbf{z}(t)}_{\mathbf{H}^{k+\frac{n+1}{2}}}\leq \frac{\delta}{N^2},\ \norm{\mathbf{u}_0-\frac{1}{T}\mathbf{\Phi}\left(\frac{\cdot}{T}\right)}_{\HH_{rad}^{s,k}}<\frac{\delta}{N^2} \quad \textit{for each }\omega\in \mathcal{N}_{\delta,N}.
		\end{align*}
		Therefore, for each $\omega\in \mathcal{N}_{\delta,N}$, there exists a time \begin{align*}
		    \tilde{T}=\tilde{T}(\omega)\in   \left[T-\frac{\delta}{N},T+\frac{\delta}{N}\right]
		\end{align*} such that the claim of \autoref{no_corrector} holds, moreover $\frac{\delta}{N}\leq \frac{T}{2} $. Since $\mathbf{\Psi}+\mathbf{\Phi}$ is a mild solution of \eqref{DPD_self_similar}, it follows from \autoref{lemma_reconstruction_solution} and uniqueness of the mild solutions for \eqref{wave_map_system} (see also \autoref{Thm_well_posed}) that, on the set $\mathcal{N}_{\delta,N}$ we have the following representation
        \begin{align}\label{representation_formula_u}
			u(t,\lvert X\rvert)=\frac{1}{\tilde{T}-t}\Phi\left(\frac{\lvert X\rvert}{\tilde{T}-t}\right)+\frac{1}{\tilde{T}-t}\Psi\left(\log\left(\frac{\tilde{T}}{\tilde{T}-t}\right),\frac{\lvert X\rvert}{\tilde{T}-t}\right)+ z(t,\lvert X\rvert),
		\end{align}
        for all $t<\tilde{T}$.
		From \eqref{representation_formula_u}, we immediately deduce that $\norm{u(t)}_{\dot{H}^{s,k}_{rad}} \to +\infty$ as $t \to \tilde{T}$. Consequently,  $\tilde{T}(\omega)$ coincides with the blowup time $\tau^*(\boldsymbol{\mathcal{U}}_0,\omega)$ for every $\omega\in \mathcal{N}_{\delta,N}.$ Hence, \begin{align*}
		    \frac{T}{2}<\tau^*(\boldsymbol{\mathcal{U}}_0,\omega)1_{\mathcal{N}_{\delta,N}}<\frac{3}{2}T,
		\end{align*} 
        and the first component of the unique corotational mild solution of \eqref{eq_intro_noise}, provided by \autoref{corollary_well_posed}, satisfies \eqref{representation_formula_calU}. To establish the final claim, we combine the representation \eqref{representation_formula_u} with the definition of $\mathcal{X}_{\delta}$ and the embdedding of $\dot{H}^{s,k}_{rad}$ in $L^{\infty}$. Indeed, on $\mathcal{N}_{\delta,N}$, we have that
		\begin{align*}
			\left\lvert (\tilde{T}-t)u(t,(\tilde{T}-t)\lvert X\rvert)-\Phi\left(\lvert X\rvert\right)\right\rvert & \leq \left\lVert  \Psi\left(\log\left(\frac{\tilde{T}}{\tilde{T}-t}\right)\right)\right\rVert_{\dot{H}^{s,k}_{rad}}+ (\tilde{T}-t)\lVert z(t,(\tilde{T}-t)\cdot)\rVert_{\dot{H}^{s,k}_{rad}} \\ & \lesssim   (\tilde{T}-t)^{\overline{\omega}}+(\tilde{T}-t)^{s-\frac{d}{2}}\rightarrow 0
		\end{align*}
        uniformly on $\mathbb{R}^d$
		as $t\rightarrow\tilde{T}^-.$ Consequently, for each $\omega\in \mathcal{N}_{\delta,N}$
		\begin{align*}
			\mathcal{U}(t,(\tau^{*}(\boldsymbol{\mathcal{U}}_0,\omega)-t)X)\rightarrow \Phi(\lvert X\rvert)X
		\end{align*}
		locally uniformly in $\mathbb{R}^d$, which completes the argument.
	\end{proof}
	\section{Irreducibility of the stochastic corotational wave maps}\label{Section_Irreducibility}
	\noindent This section is devoted to establishing the second key component in the proof of \autoref{Thm_well_posed}, i.e., the following result.
	\begin{theorem}\label{Thm_irreducibility}
		Suppose that \autoref{hp_noise_2} holds and let \begin{align*}
		    \mathbf{u}_0\in \dot{H}^{s-1,k+2}_{rad}\times \dot{H}^{s-1,k}_{rad}\subseteq \HH_{rad}^{s,k} ,\quad  \mathbf{u}_1\in \HH_{rad}^{s,k},\quad \text{and} \quad t,\eps>0
		\end{align*} Let $\mathbf{u}$ denote the unique solution of \eqref{wave_map_system} with initial condition $\mathbf{u}_0$ given by \autoref{Thm_well_posed}. Then
		\begin{align*}
			\mathbb{P}(\tau^*(\mathbf{u}_0)\geq t,~ \norm{\mathbf{u}(t)-\mathbf{u}_1}_{\HH_{rad}^{s,k}}<\eps)>0.
		\end{align*}
	\end{theorem}
	\noindent This result establishes the irreducibility of the stochastic system \eqref{wave_map_system}. As a direct consequence, it implies that, with positive probability, the blowup of $\mathbf{u}$ can be prevented up to time $t$. The proof proceeds in two steps: (i) an approximate controllability argument, and (ii) continuous dependence of the solution map on the controls.
	\subsection{Approximate controllability}\label{approx_control}
	We study the translated nonlinear wave map
\begin{align}\label{translated_wave}
		\begin{cases}
			\partial_t^2 w-\Delta w=\frac{n-3}{2\lvert x\rvert^3}(2\lvert x\rvert (w+z)-\sin (2\lvert x\rvert (w+z))),\\
			w(0)= u_0,\\
			\partial_t w(0)=\hat{u}_0,    
		\end{cases}
	\end{align}
	where $z$ is, as before, the stochastic convolution associated with the wave equation. We adopt \eqref{translated_wave} instead of the more standard formulation
	\begin{align}\label{aux_wave_map}
		\begin{cases}
			\partial_t^2 u-\Delta u=\frac{n-3}{2\lvert x\rvert^3}(2\lvert x\rvert u-\sin (2\lvert x\rvert u))+\partial_t f,\\
			u(0)= u_0,\\
			\partial_t u(0)=\hat{u}_0,   
		\end{cases}
	\end{align}
	since $w$ enjoys better continuity properties with respect to $z$ than $u$ does with respect to $f$.
    
	For convenience, we define, for any $T_1>0$,
	\begin{align*}
		\mathcal{E}^{s,k}_{0,T_1}&:=\left\{\mathbf{z}=(z,\hat z)\in C([0,T_1];\HH_{rad}^{s,k}): z(0)=\hat{z}(0)=0,\ z(t)=\int_0^t \hat{z}(s)ds \right\}.
	\end{align*}
	This is a separable Banach space if endowed with the topology of $C([0,T_1];\HH_{rad}^{s,k})$. Here, the equality $z(t)=\int_0^t \hat{z}(s)ds$ has to be understood in the sense of tempered distributions, i.e. in $\mathscr{S}'$. 
	Now we can prove the approximate controllability result for equation \eqref{translated_wave}.
	\begin{proposition}\label{prop_aux_control}
		For any $T_1>0\ \text{and } \mathbf{u}_0=(u_0,\hat{u}_0),\ \mathbf{u}_1=(u_1,\hat{u}_1)\in \dot{H}^{s-1,k+2}_{rad}\times \dot{H}^{s-1,k}_{rad}\subseteq \HH_{rad}^{s,k}$, there exists $\mathbf{z}\in \mathcal{E}^{s,k}_{0,T_1}$ such that $\mathbf{w}=(w,\partial_t w)$ solving \eqref{translated_wave} exists on $[0,T_1]$ and satisfies \begin{align*}
			\mathbf{w}(0)=\mathbf{u}_0,\quad
			\mathbf{w}(T_1)+\mathbf{z}(T_1)=\mathbf{u}_1.
		\end{align*}
	\end{proposition}
	\begin{proof}
		Without loss of generality, let us assume $T_1=1$.
		We aim to construct $\mathbf{u}=(u,\partial_t u)$ such that \begin{align*}
			u(0)=u_{0},\ u(1)=u_1,\ \partial_t u(0)=\hat{u}_0,\ \partial_t u(1)=\hat{u}_{1}.
		\end{align*}
		Since four conditions must be satisfied, it is clear that if $u$ is taken as a cubic polynomial in time, the coefficients can then be chosen to meet all the conditions. However, in order to relax the regularity assumptions on $\mathbf{u}_0$ and $\mathbf{u}_1$, and to construct $\mathbf{z}$ and $\mathbf{w}$ such that $\mathbf{u} = \mathbf{z} + \mathbf{w}$, with both $\mathbf{z}$ and $\mathbf{w}$ satisfying the requirements of the statement, it is necessary to modify the polynomial behavior through a suitable regularization procedure. With this in mind, and with a slight abuse of notation, let $P(t)$ denote the semigroup on $\dot{H}^l_{rad}$ associated to the infinitesimal generator $\Delta-I:D(\Delta-I)=\dot{H}^l_{rad}\cap \dot{H}^{l+2}_{rad}\subseteq \dot{H}^l_{rad}\rightarrow \dot{H}^l_{rad}$ for each $l>-\frac{n}{2}$.
		Let 
		\begin{align*}
			u(t)=\, &\left(1+2t^3-3t^2\right)u_0+t^2(3-2t)u_1\\ & -t^2(3-2t)\left((I-\Delta)^{-1}-(I-\Delta)^{-2}+(I-\Delta)^{-2}P(1)\right)\hat{u}_0\\ & -(I-\Delta)^{-1}((1-t)P(t)-I-(I-\Delta)^{-1}P(t)+(I-\Delta)^{-1})\hat{u}_0\\ & +t^2(3-2t)\left((I-\Delta)^{-1}+(I-\Delta)^{-2}-(I-\Delta)^{-2}P(1)\right)\hat{u}_{1}\\\ & -(I-\Delta)^{-1}\left(tP(1-t)+(I-\Delta)^{-1}P(1-t)-(I-\Delta)^{-1}P(1)\right)\hat{u}_{1},\\
			\hat{u}(t)&=\partial_t u(t),\ \mathbf{u}(t)=(u,\hat{u})(t).
		\end{align*}
		Then, clearly, 
		   $ \mathbf{u}(0)=\mathbf{u}_0,\
		\mathbf{u}(1)=\mathbf{u}_1.
		$		Moreover, we have the regularity \begin{align*}
			u\in C([0,1];\dot{H}^{s-1,k+2}_{rad}),\quad \hat{u}\in C([0,1];\dot{H}^{s-1,k}_{rad}).
		\end{align*}
		Next, define \begin{align*}
			f(t)&=\hat{u}(t)-\hat{u}_0-\int_0^t\left(\Delta u(s)+n_0(u(s))\right) ds,\quad \mathbf{F}(t)=\left(0,f(t)\right).
		\end{align*}
		By the regularity of $u, \hat{u}$, together with \autoref{lemma_regularity_nonlinear}, we have $f\in C([0,1];\dot{H}^{s-1,k}_{rad} )$. Therefore
		\begin{align*}
			\mathbf{F}\in C([0,1];\dot{H}^{s-1,k+2}_{rad}\times \dot{H}^{s-1,k}_{rad}).
		\end{align*}
		Now set 
		\begin{align*}
			\mathbf{z}(t)&=\mathbf{F}(t)+\int_0^t \mathbf{T}(t-s)\mathbf{A} \mathbf{F}(s) ds.
		\end{align*}
		It is straightforward to check that $\mathbf{z}$ satisfies
		\begin{align*}
			\begin{cases}
				\partial_t \mathbf{z}=\mathbf{A} \mathbf{z}+\partial_t \mathbf{F}\\
				\mathbf{z}(0)=0,
			\end{cases}     
		\end{align*}
		and that $\mathbf{u}$ together with $\mathbf{F}$ solves equation \eqref{aux_wave_map}. In particular, due to the smoothness of $\mathbf{F}$, we have \begin{align*}
		    \mathbf{z}\in C([0,1];\HH_{rad}^{s,k}).
		\end{align*}The claim follows.
	\end{proof}
	\subsection{Continuity of the solution mapping along the controllers}\label{Data_Continuity}
	The main result of this section is the following.
	\begin{proposition}\label{continuity_no_blow_up}
		Let $\mathbf{u}_0\in \HH_{rad}^{s,k}$ and $ \mathbf{z}\in \mathcal{E}^{s,k}_{0,T_1}$ be such that there exists a unique solution $\mathbf{w}\in C([0,T_1];\HH_{rad}^{s,k} )$ to \eqref{translated_wave}. Then, there exist neighborhoods $\mathcal{H}_0 \subseteq \HH_{rad}^{s,k}$ of $\mathbf{u}_0$ and $\mathcal{Z}_0 \subseteq \mathcal{E}^{s,k}_{0,T_1}$ of $\mathbf{z}$ such that for any $\mathbf{\Tilde{u}}_0\in \mathcal{H}_0 $ and $\mathbf{\Tilde{z}}\in \mathcal{Z}_0$, there exists a unique solution $\mathbf{\tilde{w}} \in C([0,T_1]; \HH_{rad}^{s,k})$ of \eqref{translated_wave}. Moreover the solution map $(\mathbf{u}_0,\mathbf{z})\rightarrow \mathbf{w}$  is continuous from $\mathcal{H}_0\times \mathcal{Z}_0$ into $C([0,T_1];\HH_{rad}^{s,k})$.
	\end{proposition}
	\begin{proof}
		We divide the proof into two steps.\\
		\emph{Step 1: Preparation.} Let $M\geq 1$ be the constant appearing in \eqref{growth_wave_semigroup}. We first show the existence of a time $T^{*}=T^{*}(r,R)$ such that for any
		\begin{align*}
			\lVert \mathbf{u}_0\rVert_{\HH_{rad}^{s,k}}\leq r,\ \lVert \mathbf{z}\rVert_{C([0,T_1];\HH_{rad}^{s,k})}\leq R,
		\end{align*}
		there exists a unique solution $\mathbf{w}$ to \eqref{translated_wave} belonging to the closed ball of radius $(M+1)r$ in $C([0, T^*]; \HH_{rad}^{s,k})$. Define the map 
		\begin{align*}
			\Gamma_{\mathbf{u}_0,\mathbf{z}} [\mathbf{h}](t)=\mathbf{T}(t)\mathbf{u}_0+\int_0^t \mathbf{T}(t-s) \mathbf{n}(\mathbf{h}(s)+\mathbf{z}(s))ds.
		\end{align*}
		By the properties of the wave semigroup on $\HH_{rad}^{s,k}$, it follows immediately that $\Gamma_{\mathbf{u}_0,\mathbf{z}} [\mathbf{h}](t)\in C([0,T_1];\HH_{rad}^{s,k})$. Moreover, applying \autoref{lemma_regularity_nonlinear}, we obtain
		\begin{align*}
			\lVert\Gamma_{\mathbf{u}_0,\mathbf{z}} [\mathbf{h}](t)\rVert_{\HH_{rad}^{s,k}}&\leq  M e^{T_1} r+C_{M,s,k}\left(R+r\right)^3(1+R^{2k}+r^{2k})(e^{T_1}-1)\\ & \leq  (M+1)r
		\end{align*}
		provided that $T_1$ is sufficiently small. Therefore for sufficiently small $T^*>0$ the operator $\Gamma_{\mathbf{u}_0,\mathbf{z}}$ maps the closed ball of radius $(M+1)r$ in $C([0,T^*];\HH_{rad}^{s,k})$ into itself. Next, we have to show that $\Gamma_{\mathbf{u}_0,\mathbf{z}}$ is a contraction. Using the properties of the wave semigroup on $\HH_{rad}^{s,k}$, we have
		\begin{align*}
			\lVert\Gamma_{\mathbf{u}_0,\mathbf{z}} [\mathbf{h}_1](t)-\Gamma_{\mathbf{u}_0,\mathbf{z}} [\mathbf{h}_2](t)\rVert_{\HH_{rad}^{s,k}}& \lesssim \lVert \mathbf{h}_1-\mathbf{h}_2\rVert_{C([0,T^*];\HH_{rad}^{s,k})} \left(R+r\right)^2(1+R^{2k}+r^{2k})(e^{T^*}-1)  \\ & \leq \frac{1}{2}\lVert\mathbf{h}_1-\mathbf{h}_2\rVert_{C([0,T^*];\HH_{rad}^{s,k})}, 
		\end{align*}
		provided that $T^*$ is sufficiently small. This choice of $T^*$ may be smaller than in the previous step, but it depends only on $r$ and $R$.\\
		Similarly, we can solve \eqref{translated_wave} on any interval $[T^{*}(r,R)k,T^{*}(r,R)(k+1)]$ for $k\in \left\{0,\dots, \lfloor \frac{T_1}{T^{*}(r,R)}\rfloor \right\}$ provided that the initial condition satisfies $\lVert \mathbf{u}_{k}\rVert_{\HH_{rad}^{s,k}}\leq r. $ To keep track of the initial condition and the time interval considered, we denote by $\mathbf{w}^{k+1}(\mathbf{z},\mathbf{u}_k,\cdot)$ the unique solution of \eqref{translated_wave} in $[T^{*}(r,R)k,T^{*}(r,R)(k+1)]$ with initial condition $\mathbf{u}_{k}. $ We now show that $\mathbf{w}^{k+1}(\mathbf{z},\mathbf{u}_k,\cdot)$ depends continuously on $\mathbf{u}_k$ and $\mathbf{z}$ in $C([T^{*}(r,R)k,$ $T^{*}(r,R)(k+1)];\HH_{rad}^{s,k})$. Indeed, since $\Gamma_{\cdot,\cdot}$ is a contraction of constant $\frac{1}{2}$ in the relevant parameter range, by \autoref{lemma_regularity_nonlinear} we have
		\begin{align*}
			&\lVert \mathbf{w}^{k+1}(\mathbf{z}_1,\mathbf{u}_{k,1},t)-\mathbf{w}^{k+1}(\mathbf{z}_2,\mathbf{u}_{k,2},t)  \rVert_{\HH_{rad}^{s,k}}\\ \leq \, & \lVert \Gamma_{\mathbf{u}_{k,1},\mathbf{z}_1} [\mathbf{w}^{k+1}(\mathbf{z}_1,\mathbf{u}_{k,1},\cdot)](t)-\Gamma_{\mathbf{u}_{k,2},\mathbf{z}_1} [\mathbf{w}^{k+1}(\mathbf{z}_1,\mathbf{u}_{k,1},\cdot)](t)  \rVert_{\HH_{rad}^{s,k}}\\ & +\lVert \Gamma_{\mathbf{u}_{k,2},\mathbf{z}_1} [\mathbf{w}^{k+1}(\mathbf{z}_1,\mathbf{u}_{k,1},\cdot)](t)-\Gamma_{\mathbf{u}_{k,2},\mathbf{z}_2} [\mathbf{w}^{k+1}(\mathbf{z}_1,\mathbf{u}_{k,1},\cdot)](t)  \rVert_{\HH_{rad}^{s,k}}\\ & +\lVert \Gamma_{\mathbf{u}_{k,2},\mathbf{z}_2} [\mathbf{w}^{k+1}(\mathbf{z}_1,\mathbf{u}_{k,1},\cdot)](t)-\Gamma_{\mathbf{u}_{k,2},\mathbf{z}_2} [\mathbf{w}^{k+1}(\mathbf{z}_2,\mathbf{u}_{k,2},\cdot)](t)  \rVert_{\HH_{rad}^{s,k}}\\  \leq \, &  C\left(\lVert \mathbf{u}_{k,1}-\mathbf{u}_{k,2}\rVert_{\HH_{rad}^{s,k}}+\lVert \mathbf{z}_1-\mathbf{z}_2\rVert_{C([0,T_1];\HH_{rad}^{s,k})}\right)\\ &+\frac{1}{2}\lVert \mathbf{w}^{k+1}(\mathbf{z}_1,\mathbf{u}_{k,1},\cdot)-\mathbf{w}^{k+1}(\mathbf{z}_2,\mathbf{u}_{k,2},\cdot)  \rVert_{C([0,T_1];\HH_{rad}^{s,k})},
		\end{align*}
		for some positive constant $C$. The latter implies
		\begin{multline}\label{continuity_estimate}
			\lVert \mathbf{w}^{k+1}(\mathbf{z}_1,\mathbf{u}_{k,1},\cdot)-\mathbf{w}^{k+1}(\mathbf{z}_2,\mathbf{u}_{k,2},\cdot)  \rVert_{C([T^{*}(r,R)k,T^{*}(r,R)(k+1)];\HH_{rad}^{s,k})}\\ \leq C\left(\lVert \mathbf{u}_{k,1}-\mathbf{u}_{k,2}\rVert_{\HH_{rad}^{s,k}}+\lVert \mathbf{z}_1-\mathbf{z}_2\rVert_{C([0,T_1];\HH_{rad}^{s,k})}\right),
		\end{multline}
		which establishes the claim.\\
		\emph{Step 2: End of the proof.} 
		Let us introduce \begin{align*}
			R=1+\lVert \mathbf{z}\rVert_{C([0,T_1];\HH_{rad}^{s,k})},\ r=1+\lVert \mathbf{w}\rVert_{C([0,T_1];\HH_{rad}^{s,k})},
		\end{align*}
        with 
		$\mathbf{z}$ and $\mathbf{w}$ the ones in the statement. Let us set \begin{align*}
		    N=\left\lfloor \frac{T_1}{T^{*}(r,R)}\right\rfloor,\quad \delta_0=\frac{1}{2^{N}(C\vee 1)^N },
		\end{align*} where $C$ is the constant in \eqref{continuity_estimate}. Define $\HH_0$ and $\mathcal{Z}_0$ as the closed balls of radius $\delta_0$ centered at $\mathbf{u}_0$  and $\mathbf{z}$, respectively. Thanks to \eqref{continuity_estimate}
		\begin{align*}
			\lVert \mathbf{w}^{1}(\Tilde{\mathbf{z}},\Tilde{\mathbf{u}}_{0},\cdot)-\mathbf{w}(\mathbf{z},\mathbf{u}_{0},\cdot)  \rVert_{C([0,T^{*}(r,R)];\HH_{rad}^{s,k})} \leq \frac{1}{2^{N-1}(C\vee 1)^{N-1}}.
		\end{align*}
		In particular 
		\begin{align*}
			\lVert \mathbf{w}^{1}(\Tilde{\mathbf{z}},\Tilde{\mathbf{u}}_{0},T^{*}(r,R))\rVert_{\HH_{rad}^{s,k}}\leq \lVert\mathbf{w}\rVert_{C([0,T_1];\HH_{rad}^{s,k})}+\frac{1}{2^{N-1}(C\vee 1)^{N-1}} <r,   
		\end{align*}
		and we can construct $\mathbf{w}^{2}(\Tilde{\mathbf{z}},\mathbf{v}^{1}(\Tilde{\mathbf{z}},\Tilde{\mathbf{u}}_{0},T^{*}(r,R)),\cdot)$ extending $\mathbf{w}^{1}(\Tilde{\mathbf{z}},\Tilde{\mathbf{u}}_{0},\cdot)$. Iterating the argument, we can extend $\mathbf{w}^{1}(\Tilde{\mathbf{z}},\Tilde{\mathbf{u}}_{0},\cdot)$ to the whole $[0,T_1]$ proving the first part of the claim and the continuity of the solution map in $(\mathbf{u}_0,\mathbf{z}).$ Replacing $(\mathbf{u}_0,\mathbf{z})$ by any element in $\mathcal{H}_0,\ \mathcal{Z}_0$ the rest of the claim follows.
	\end{proof}
	As a direct consequence of \autoref{prop_aux_control} and \autoref{continuity_no_blow_up} we obtain the following proof.
	\begin{proof}[Proof of \autoref{Thm_irreducibility}]
		Let us start fixing  $\mathbf{\overline{u}}_1\in \mathscr{S}_{rad}\times \mathscr{S}_{rad} $ such that \begin{align*}
			\norm{\mathbf{\overline{u}}_1-\mathbf{u}_1}_{\HH_{rad}^{s,k}}<\frac{\eps}{2}.
		\end{align*}
		It then follows that
		\begin{align*}
			\mathbb{P}\left(\tau^*(\mathbf{u}_0)\geq t,\quad \norm{\mathbf{u}(t)-\mathbf{u}_1}_{\HH_{rad}^{s,k}}<\eps\right)\geq \mathbb{P}\left(\tau^*(\mathbf{u}_0)\geq t,\quad \norm{\mathbf{u}(t)-\mathbf{\overline{u}}_1}_{\HH_{rad}^{s,k}}<\frac{\eps}{2}\right).
		\end{align*}
		By \autoref{prop_aux_control} with $T_1=t$, there exists $\mathbf{z}\in \mathcal{E}^{s,k}_{0,T_1}$ such that \begin{align*}
		    \mathbf{w}(t)+\mathbf{z}(t)=\mathbf{\overline{u}}_1.
		\end{align*} Furthermore, by \autoref{continuity_no_blow_up}, there exists a neighborhood $\mathcal{Z}_0$ of $\mathbf{z}$ in $\mathcal{E}^{s,k}_{0,T_1}$,  such that for every $\mathbf{\overline{z}}\in \mathcal{Z}_0 $
        \begin{align*}
        \norm{\mathbf{\overline{w}}(t)+\mathbf{\overline{z}}(t)-\mathbf{\overline{u}}_1}<\frac{\eps}{2}.
        \end{align*}
        Consequently, we have
		\begin{align*}
			\mathbb{P}\left(\tau^*(\mathbf{u}_0)\geq t,\quad\norm{\mathbf{u}(t)-\mathbf{\overline{u}}_1}_{\HH_{rad}^{s,k}}<\frac{\eps}{2}\right)\geq \mathbb{P}(\mathbf{z}\in \mathcal{Z}_0).    
		\end{align*}
		The latter probability is strictly positive due to the non-degeneracy of the noise $\mathcal{W}$, its regularity and \autoref{lemma_support_z}.
	\end{proof}
	\section{Proof of \autoref{main_thm}}
    
	\noindent As announced in \autoref{sec_strategy}, the proof follows by combining \autoref{Thm_irreducibility} and \autoref{stability_lemma}. For a given $\mathcal{T}>0$, let us fix \begin{align*}
	    T=\frac{\mathcal{T}}{3}
	\end{align*} and subsequently $ \eps>0$ small enough such that  \autoref{stability_lemma} applies with this choice of $T.$ Define \begin{align*}
	    \Omega_0=\left\{\omega\in \Omega:\ \tau^{*}(\boldsymbol{\mathcal{U}}_0)>\frac{\mathcal{T}}{2},\quad \norm{\mathbf{u}\left(\frac{\mathcal{T}}{2},\cdot\right)-\frac{1}{T}\mathbf{\Phi}\left(\frac{\cdot}{T}\right)}_{\HH_{rad}^{s,k}}<\frac{\eps}{1\vee C_{d,s,k}}\right\},
	\end{align*}where $C_{d,s,k}$ denotes the maximum of the hidden constants appearing in \cite[Proposition A.5, Remark A.6]{glogic2022stable} applied for $s,s-1,k,k-1$. Then, on $\Omega_0$ we also have
	\begin{align*}
		\norm{\boldsymbol{\mathcal{U}}\left(\frac{\mathcal{T}}{2},\cdot\right)-\boldsymbol{\mathcal{U}}_{\frac{\mathcal{T}}{3}}(0,\cdot)}_{\HH^{s,k}}<\eps.
	\end{align*}
	By \autoref{Thm_irreducibility}, $\mathbb{P}(\Omega_0)>0$, and clearly $\Omega_0\in \mathcal{F}_{\mathcal{T}/2}$. Let us define \begin{align*}
		\mathbf{\Tilde{u}}_0=\begin{cases}
			\mathbf{u}(\frac{\mathcal{T}}{2})\quad &\text{if }\omega\in \Omega_0\\
			0\quad &\text{otherwise.}
		\end{cases}
	\end{align*}
	By \autoref{stability_lemma}, the solution of equation \eqref{eq_intro_noise} with initial condition $\boldsymbol{\tilde{\mathcal{U}}}_0(X)=\mathbf{\Tilde{u}}_0(\lvert X\rvert) X$ and Brownian motion $\Tilde{\mathcal{W}}_t=\mathcal{W}_{\frac{\mathcal{T}}{2}+t}-\mathcal{W}_{\frac{\mathcal{T}}{2}}$, denoted by $\boldsymbol{\mathcal{\tilde{U}}},$ blows up in a self-similar way before $\frac{3}{2}T=\frac{\mathcal{T}}{2} $ with positive probability.
	Due to pathwise uniqueness of solutions for \eqref{eq_intro_noise} guaranteed by \autoref{corollary_well_posed}, it follows that $\boldsymbol{\mathcal{U}}(t+\frac{\mathcal{T}}{2})$ coincides with $\boldsymbol{\Tilde{\mathcal{U}}}(t)$ on $\Omega_0.$ Consequently, $\boldsymbol{\mathcal{U}}(t)$ also blows up in a self-similar way with positive probability before $\frac{\mathcal{T}}{2}+\frac{3}{2}T=\mathcal{T}$. Finally, defining \begin{align*}
	    \Upsilon\left(\frac{\mathcal{T}}{2}+t,x\right)=\Psi\left(\log\left(\frac{\tau^{*}(\boldsymbol{\mathcal{U}}_0)-\frac{\mathcal{T}}{2}}{\tau^{*}(\boldsymbol{\mathcal{U}}_0)-\frac{\mathcal{T}}{2}-t}\right),x\right),
	\end{align*} the proof is complete.

	\bibliography{demo.bib}
	\bibliographystyle{plain}
\end{document}